\documentclass[12pt,english]{article}
\usepackage{babel}
\usepackage[cp1250]{inputenc}
\usepackage[T1]{fontenc}
\usepackage{amsfonts}
\usepackage{amsmath}
\usepackage{amsthm}
\usepackage{mathrsfs}
\usepackage{hyperref}
\usepackage{enumerate}
\usepackage{graphicx}
\usepackage{pictex}
\usepackage{color}

\newcommand{\eqq}[2]{\begin{equation}  #1  \label{#2} \end{equation}    }
\newcommand*{\norm}[1]{\left\Vert{#1}\right\Vert}
\newcommand*{\abs}[1]{\left\vert{#1}\right\vert}

\newcommand*{\izj}{\int_{0}^{1}}

\newcommand*{\p}{\partial}
\newcommand*{\poch}{\frac{\p}{\p x}}

\newcommand*{\ia}{I^{\alpha}}
\newcommand*{\al}{\alpha}
\newcommand*{\uz}{u_{0}}
\newcommand*{\tj}{t_{1}}

\newcommand*{\vf}{\varphi}
\newcommand*{\ve}{\varepsilon}
\newcommand{\podst}[2]{\left\{ \begin{array}{ll} #1 \\  #2 \\  \end{array} \right\}}

\newcommand{\hd}{\hspace{0.2cm}}
\newcommand{\no}{\noindent}
\newcommand{\m}[1]{\mbox{#1}}

\newtheorem{rem}{{\textbf {Remark}}}
\newtheorem{lem}{{\textbf {Lemma}}}
\newtheorem{prop}{{\textbf {Proposition}}}
\newtheorem{theorem}{\textbf {Theorem}}
\newtheorem{coro}{\textbf  {Corollary} }
\newtheorem{de}{\textbf  {Definition} }

\newcommand{\izx}{\int_{0}^{x}}

\newcommand{\da}{D^{\alpha}}

\newcommand{\izt}{\int_{0}^{t}}

\newcommand{\ld}{L^{2}(0,1)}

\newcommand{\jgja}{\frac{1}{\Gamma(1-\al)}}

\newcommand{\ca}{c_{\al}}

\newcommand{\dd}{\mathcal{D}_{\al}}
\newcommand{\xj}{x_{1}}
\newcommand{\tz}{t_{0}}
\newcommand{\sj}{s_{1}}
\newcommand{\sd}{s_{2}}

\newcommand{\sde}{s_{\delta}}
\newcommand{\uj}{u^{1}}
\newcommand{\ude}{u^{\delta}}
\newcommand{\ud}{u^{2}}
\newcommand{\uzj}{u_{0}^{1}}
\newcommand{\uzd}{u_{0}^{2}}
\newcommand{\uzde}{\uz^{\delta}}

\begin{document}
\title{\bf A space - fractional Stefan problem}

\author{Katarzyna Ryszewska\\
\small{Department of Mathematics and Information Sciences, Warsaw University of Technology,} \\
\small{Koszykowa 75, 00-662 Warsaw,
Poland,}\\
\small{E-mail address: K.Ryszewska@mini.pw.edu.pl}}

\date{}

\maketitle

\abstract{We study a space-fractional Stefan problem, where the non-local diffusion flux is modeled by the Caputo derivative. We obtain the unique existence of classical solution to this problem.}
\vspace{0.4cm}

\no Keywords: fractional Stefan problem; Caputo derivative; analytic semigroup theory.
\vspace{0.2cm}

\section{Introduction}
In this paper we study the fractional in space, one-phase one-dimensional Stefan problem. The classical Stefan model describes the process of changing the phase in homogenous medium, for example, the ice melting.

 Recently, the diffusion in non-homogenous media is  frequently described by the models involving the fractional derivatives. In the case of phenomena exhibiting memory effects, the anomalous super-diffusion may be modeled with a use of time-fractional derivatives. There exists a number of papers which provide partial results concerning Stefan problem where the time derivative is replaced by the fractional Caputo derivative. In \cite{KRR} the existence of weak solutions in non cylindrical domain with fixed boundary was shown. This problem was also considered in papers \cite{Hopf},  \cite{Ros} and  \cite{exa1}. In the first one, under suitable regularity assumptions the Hopf lemma was proven, while in the second and third one the authors obtained some exact solutions.
 The issue of establishing the appropriate time-fractional Stefan model appears to be challenging itself. In the paper \cite{Vzl} the author presented the idea of representing the difussive flux in the form of the Riemann-Liouville time-fractional derivative of the spatial gradient of solution. Recently, several papers based on this proposal where published. In \cite{Cere}, a few Stefan models concerning sharp interphase as well as diffusive interphase were proposed. Further discussion was made in \cite{Rosa} - \cite{Rosd} where the authors consider the issue of proper formulation of the model and the existence of special solution.

This paper is devoted to the moving boundary problem with a non-locality in space. The idea of modelling anomalous sub-diffusion with a use of space-fractional operators was introduced in \cite{V2}. In this paper, the author considered the model of infiltration of water into heterogeneous soil. In the homogenous media, there is a well known constitutive relation that associates the diffusive flux with a gradient of solution. In \cite{V2} the author presented a model where the diffusive flux takes the form of the fractional Caputo derivative.

In this paper we will present the mathematical analysis of the following free boundary problem
\begin{equation}\label{Stefan}
 \left\{ \begin{array}{ll}
u_{t} - \poch \da u = 0 & \textrm{ in } \{(x,t):0<x<s(t), 0<t<T\}=: Q_{s,T}, \\
u_{x}(0,t) =0, \ \ u(t,s(t)) = 0 & \textrm{ for  } t \in (0,T), \\
u(x,0) = u_{0}(x) & \textrm{ for } 0<x<s(0)=b, \\
\dot{s}(t) = -(\da u)(s(t),t) & \textrm{ for  } t \in (0,T),
\end{array} \right. \end{equation}
where $\da$ denotes the fractional Caputo derivative with respect to the spatial variable, that is
\[
\da u(x,t) = \jgja \poch \izx (x-p)^{-\al} [u(p,t)-u(0,t)]dp.
\]
The above model was proposed in \cite{V1}. The physical motivation to study (\ref{Stefan}) originates from anomalous sub-diffusion processes such as mass transport or solidification of substances in a complex media.
We note that, here $u$ is a function of two variables and can be regarded as a temperature of medium or a density of transported substance, while $s:[0,T]\rightarrow \mathbb{R}$  is a moving part of the boundary of the domain. We assume that the constants $T$ and $b$ are positive and fixed and $\uz$ is given nonnegative initial condition. We note that here we impose the zero Neumann condition at the left boundary.

The main results of this paper (established in Theorem \ref{final} and Corollary 4) claim that if we assume that the initial value $\uz$ is regular enough, then the system (\ref{Stefan}) has exactly one classical solution.
The proof partly relies on the recent result obtained in~\cite{ja}, where the unique existence of the solution to the linear initial-value-boundary problem
\begin{equation}\label{aa}
 \left\{ \begin{array}{ll}
u_{t} - \poch \da u = f & \textrm{ in } (0,1) \times (0,T),\\
u_{x}(0,t) = 0, \ \ u(1,t) = 0 & \textrm{ for  } t \in (0,T), \\
u(x,0) = u_{0}(x) & \textrm{ in } (0,1) \\
\end{array} \right. \end{equation}
was proven. Precisely, there was shown that if we consider the operator $\poch \da$ on the appropriate domain, then it is a generator of an analytic semigroup.

Our approach follows the standard methods for solving the classical Stefan problem, presented in \cite{Andreucci}. First of all, we focus our attention on the problem considered in a non cylindrical domain with a given function $s$. We apply the transformation into the cylindrical domain and we find a regular solution by means of the abstract evolution operator theory. Then, we prove the weak extremum principle and the space-fractional version of Hopf lemma, i.e. $(\da u)(s(t),t) < 0 \hd \forall \hspace{1 mm} t \in (0,T]$. Finally, by the Schauder fixed point theorem, we are able to obtain the existence of a pair $(u,s)$ which is a classical solution to (\ref{Stefan}). At last, we prove the monotone dependence upon data in order to obtain the uniqueness of the solution.

The paper is organized as follows. For the readers convenience, in  Chapter~2, we will briefly recall selected theorems from fractional calculus and the results from~\cite{ja} which play a significant role in the further part of the paper. Chapter 3 is devoted to the non cylindrical linear problem with a given boundary. At last, in Chapter~4, we state the extremum principles and the proof of the main result.

\section{Preliminary results}
We begin with recalling the definitions of fractional operators. For an introduction to theory of fractional  operators we refer to \cite{Kilbas} and \cite{Samko}.
\begin{de}\label{fracdef}
Let $L, \al > 0$. For $f \in L^{1}(0,L)$ we introduce the fractional integral $\ia$ by the formula
\[
\ia f (x) = \frac{1}{\Gamma(\al)}\izx (x-p)^{\al-1}f(p)dp.
\]
For $\al \in (0,1)$ and $f$ regular enough, we define the Riemann-Liouville fractional derivative
\[
\p^{\al}f(x) =  \poch I^{1-\al}f(x)= \jgja \poch \izx (x-p)^{-\al} f(p)dp
\]
and the Caputo fractional derivative
\[
\da f(x) = \poch I^{1-\al}[f(\cdot)-f(0)](x)=  \jgja \poch \izx (x-p)^{-\al}[f(p)-f(0)]dp.
\]
\end{de}
\no We will introduce the characterization of the domain of the Riemann-Liouville derivative, considered as an operator acting on $L^{2}(0,1)$. This characterization plays a significant role in establishing the properties of $\poch \da$ in the semigroup framework. The following proposition is the extended version of \cite[Theorem 2.1]{Y} which can be found in the Appendix of \cite{KY}.
\begin{prop}\label{eq}
Let us denote by $H^{\al}(0,1)$ the fractional Sobolev space (see \cite[definition 9.1]{Lions}).
We define a subspace of $H^{\al}(0,1)$ as follows
\[
{}_{0}H^{\al}(0,1)=\left\{
\begin{array}{lll}
H^{\al}(0,1)  & \m{ for } & \al\in (0,\frac{1}{2}), \\
\{ u \in H^{\frac{1}{2}}(0,1): \hd \int_{0}^{1}\frac{|u(x)|^{2}}{x}dx<\infty \} & \m{ for } & \al=\frac{1}{2}, \\
\{ u \in H^{\al}(0,1): \hd u(0)=0\} & \m{ for } & \al\in (\frac{1}{2},1).
\end{array}
\right.
\]\
We equip ${}_{0}H^{\al}(0,1)$ with the norm induced from $H^{\al}(0,1)$ for $\al \neq \frac{1}{2}$ and we set
\[
\norm{u}_{_{0}H^{\frac{1}{2}}(0,1)}  =\left(\norm{u}_{H^{\frac{1}{2}}(0,1)}^{2} + \izj \frac{\abs{u(x)}^{2}}{x} dx\right)^{\frac{1}{2}}.
\]
Then, for $\al\in (0,1)$ the operators
$\ia: L^{2}(0,1)\longrightarrow {}_{0}H^{\al}(0,1)$ and $\p^{\al}:{}_{0}H^{\al}(0,1) \longrightarrow L^{2}(0,1)$ are isomorphism and the following
inequalities
$$\ca^{-1} \| u \|_{_{0}H^{\al}(0,1)}
\leq \| \p^{\al} u \|_{L^{2}(0,1)}\leq \ca
\| u \|_{_{0}H^{\al}(0,1)} \hd \m{ for } u \in {}_{0}H^{\al}(0,1),
$$

$$
\ca^{-1} \| \ia f  \|_{_{0}H^{\al}(0,1)}
\leq \|  f \|_{L^{2}(0,1)}\leq \ca \| \ia f
\|_{_{0}H^{\al}(0,1)} \hd \m{ for } f  \in L^{2}(0,1)
$$
hold, where $\ca$ denotes a positive constant dependent on $\al$.
\end{prop}

\begin{coro}\label{eqc}
For $\al, \beta \in (0,1]$ there holds $I^{\beta}: {}_{0}H^{\al}(0,1)\rightarrow {}_{0}H^{\al+\beta}(0,1),$
where in the case $\al+\beta > 1$
\[
_{0}H^{\al+\beta}(0,1) = \{f \in H^{\al+\beta}(0,1): f(0) = 0,\hd f' \in {}_{0}H^{\al+\beta-1}(0,1) \}.
\]
Furthermore, there exists a positive constant $c$ depending only on $\al,\beta$ such that for every $f \in {}_{0}H^{\al}(0,1) $
\[
\norm{I^{\beta}f}_{{}_{0}H^{\al+\beta}(0,1)} \leq c \norm{f}_{{}_{0}H^{\al}(0,1)}.
\]

\end{coro}
\begin{proof}
It is an easy consequence of Proposition \ref{eq}. If $f \in {}_{0}H^{\al}(0,1)$ then, by Proposition \ref{eq}, there exists $w \in L^{2}(0,1)$ such that $f = I^{\al}w$. Hence,
\[
I^{\beta} f = I^{\beta} I^{\al}w = I^{\al+\beta}w.
\]
If $\al+\beta \leq 1$, then applying again Proposition \ref{eq} we obtain that $I^{\beta} f \in {}_{0}H^{\al+\beta}(0,1).$ Furthermore, we note that $w = \p^{\al}f$ and by Proposition \ref{eq} we have
\[
\norm{I^{\beta}f}_{{}_{0}H^{\al+\beta}(0,1)} = \norm{I^{\al+\beta}w}_{{}_{0}H^{\al+\beta}(0,1)} \leq c(\al,\beta) \norm{w}_{L^{2}(0,1)} \leq c(\al,\beta) \norm{f}_{{}_{0}H^{\al}(0,1)}.
\]

In the case $1<\al+\beta\leq 2$ we note that $(I^{\beta} f)(0) = (I^{\al+\beta}w)(0) = 0$ and $\poch I^{\beta}f = I^{\al+\beta-1}w \in {}_{0}H^{\al+\beta-1}(0,1)$, which leads to $I^{\beta} f \in {}_{0}H^{\al+\beta}(0,1).$\\
 In order to show the norm estimate we use the Poincar\`{e} inequality together with Proposition \ref{eq} to obtain
\[
\norm{I^{\beta}f}_{_{0}H^{\al+\beta}(0,1)} = \norm{I^{\al+\beta}w}_{_{0}H^{\al+\beta}(0,1)} \leq c
\norm{I^{\al+\beta-1}w}_{_{0}H^{\al+\beta-1}(0,1)}
\]
\[
\leq c(\al,\beta) \norm{w}_{L^{2}(0,1)} \leq c(\al,\beta) \norm{f}_{{}_{0}H^{\al}(0,1)}.
\]
\end{proof}

Now, we will briefly recall the ideas introduced in \cite{ja}. Namely, we will look at the operator $\poch \da$ from the perspective of operator theory.\\
\no At first, we need to characterize the domain of the operator $\poch \da$. We note that for absolutely continuous $u$ the Caputo fractional derivative may be equivalently written in the form $\da u = I^{1-\al}u_{x}$. Thus, just by definition we have $\poch \da u = \poch I^{1-\al} u_{x} = \p^{\al}u_{x},$ whenever one of the sides of identity is meaningful.
From Proposition \ref{eq} the domain of $\p^{\al}$ in $\ld$ coincides with ${}_{0}H^{\al}(0,1).$ Thus we may consider the domain of $\poch \da$ as $\{u\in H^{1+\al}(0,1): u_{x} \in {}_{0}H^{\al}(0,1)\}$. Taking into account the boundary condition in (\ref{aa}) we finally define the domain of $\poch \da$ as
\eqq{D(\poch \da) \equiv \dd := \{u \in H^{1+\al}(0,1) :u_{x} \in {}_{0}H^{\al}(0,1),\hd  u(1)=0\}.}{dziedzina}
We equip $\dd$ with the graph norm
\[
\norm{u}_{\dd} = \norm{u}_{H^{1+\al}(0,1)} \m{ for } \al \in (0,1) \setminus \{\frac{1}{2}\}
\]
and
\[
\norm{u}_{\dd} = \left( \norm{u}_{H^{\frac{3}{2}}(0,1)}^{2} + \izj \frac{\abs{u_{x}(x)}^{2}}{x}dx\right)^{\frac{1}{2}} \m{ for }\al = \frac{1}{2}.
\]
We finish this chapter by recalling the main result from \cite{ja} which will play a crucial role here.

\begin{theorem}\cite[Theorem 2]{ja}\label{analitycznosc}
Operator $\poch \da: \dd \subseteq \ld \rightarrow \ld$ is a generator of analytic semigroup.
\end{theorem}

\section{Solution to (\ref{Stefan}) with a given function $s$.}

In this chapter we turn our attention to the system (\ref{Stefan}). At first, we will find a regular solution to (\ref{Stefan}) assuming that function $s$ is given. That is, we will search for the solution to

\begin{equation}\label{niech}
 \left\{ \begin{array}{ll}
u_{t} - \poch \da u = 0 & \textrm{ in }  Q_{s,T}, \\
u_{x}(0,t) = 0, \ \ u(t,s(t)) = 0 & \textrm{ for  } t \in (0,T), \\
u(x,0) = u_{0}(x) & \textrm{ for } 0<x<b \\
\end{array} \right. \end{equation}
with a given function $s:[0,T]\rightarrow \mathbb{R}$. We assume that
\eqq{s \in C[0,T],\hd s(0)=b, \hd \exists \hd  M > 0 \hd \m{such that } 0 < \dot{s}(t) \leq M \m{ for every } t \in [0,T].
}{zals}

At first, we may assume that the initial condition is square integrable. We will specify the assumptions on $\uz$ in detail, as we will formulate the main theorems.

\subsection{Transformation to the cylindrical domain}

First of all, we will change the coordinates in order to pass to the cylindrical domain. We apply the standard substitution $p = \frac{x}{s(t)}$ and we define
\eqq{v(p,t) := u(s(t)p,t) = u(x,t).}{defv}
We will write the system (\ref{niech}) in terms of $v$. Firstly, we note that $\frac{\p}{\p p} = s(t)\poch$, thus
\[
v_{p}(p,t) = \frac{\p}{\p p}v(p,t) = \frac{\p}{\p p}u(s(t)p,t) = s(t)\poch u(s(t)p,t) = s(t)u_{x}(x,t),
\]
\[
v_{t}(p,t) = \frac{d}{dt}u(s(t)p,t) =u_{t}(x,t) + p\dot{s}(t)u_{x}(x,t).
\]
Together we have
\[
u_{t}(x,t) = v_{t}(p,t) - p\frac{\dot{s}(t)}{s(t)}v_{p}(p,t).
\]
Furthermore, since $v_{r}(r,t) = s(t)u_{x}(s(t)r,t)$, we may write
\[
\Gamma(1-\al)(\p^{\al}v_{p})(p,t) = \frac{\p}{\p p}\int_{0}^{p}(p-r)^{-\al}v_{r}(r,t)dr
=s(t)\frac{\p}{\p p}\int_{0}^{p}(p-r)^{-\al}u_{x}(s(t)r,t)dr
\]
\[
= \podst{s(t)r = w}{s(t)dr = dw}
 =  \frac{\p}{\p p}\int_{0}^{s(t)p}(p-\frac{w}{s(t)})^{-\al}u_{x}(w,t)dw
 \]
 \[
=s^{\al}(t)\frac{\p}{\p p}\int_{0}^{s(t)p}(s(t)p-w)^{-\al}u_{x}(w,t)dw =s^{\al+1}(t)\poch\izx (x-w)^{-\al}u_{x}(w,t)dw.
\]
In this way we obtained that
\eqq{
(\p^{\al}u_{x})(x,t)=\frac{1}{s^{1+\al}(t)}(\p^{\al}v_{p})(p,t).
}{zamda}
Denoting
\eqq{
v_{0}(p) = u_{0}(pb)}{defv0}
and renaming $p$ by $x$ we obtain that $v$ satisfies

\begin{equation}\label{niecv}
 \left\{ \begin{array}{ll}
v_{t} - x\frac{\dot{s}(t)}{s(t)}v_{x} - \frac{1}{s^{1+\al}(t)} \poch \da v = 0 & \textrm{ for } 0<x<1, 0<t<T,\\
v_{x}(0,t) = 0, \ \ v(1,t) = 0 & \textrm{ for  } t \in (0,T), \\
v(x,0) = v_{0}(x) & \textrm{ for } 0<x<1. \\
\end{array} \right. \end{equation}

\no In the next section we will find a unique solution to (\ref{niecv}) which will have appropriate regularity properties.

\subsection{Solution to transformed problem}
We will solve the system (\ref{niecv}) by means of the theory of evolution operators.
Let us define the family of operators $A(\cdot):\dd\subseteq L^{2}(0,1)\rightarrow L^{2}(0,1)$ given by the following formula
\eqq{
A(t) =   x\frac{\dot{s}(t)}{s(t)}\poch + \frac{1}{s^{1+\al}(t)} \poch \da.
}{At}
Let us denote
\[
A_{1}(t) =  x\frac{\dot{s}(t)}{s(t)}\poch \m{ and } A_{2}(t) = \frac{1}{s^{1+\al}(t)} \poch \da
\]

From Theorem \ref{analitycznosc} and assumption (\ref{zals}) we may infer that the family $A_{2}(\cdot)$ satisfies the assumption (6.1.1) from \cite{Lunardi}. Namely,
\[
\forall t \in [0,T] \hd A_{2}(t) \m{ is sectorial and } D(A_{2}(t)) \equiv \dd.
\]
\eqq{
t\mapsto A_{2}(t) \in C^{0,1}([0,T];B(\dd,L^{2}(0,1))).
}{Aholder}
However, since $\dot{s}$ is not H\"older continuous we are not allowed to use directly the results from \cite[Chapter 6]{Lunardi} to the family $A(\cdot)$. Hence, we are going to find firstly a mild solution to the problem (\ref{niecv}). Then we will show that this mild solution actually satisfies (\ref{niecv}) almost everywhere. Finally, we will further increase the regularity of the solution.
Let us denote by $\{G(t,\sigma): 0\leq \sigma\leq t \leq T\}$  the evolution operator associated with $A_{2}(t)$ and let us denote by $c$ a generic positive constant which is a continuous increasing function of $T$.
Since $A_{2}(t)$ fulfills the condition (\ref{Aholder}) then by \cite[Corollary 6.1.8]{Lunardi} we obtain that for every $\theta, \delta \in (0,1)$, $G$ satisfies the following estimates. If $g\in [L^{2}(0,1),\dd]_{\delta}$, then for any $0\leq \sigma < t \leq T$
\eqq{
\norm{G(t,\sigma)g}_{\dd} \leq \frac{c}{(t-\sigma)^{1-\delta}}\norm{g}_{[L^{2}(0,1),\dd]_{\delta}}.
}{r1}
Moreover, for any $0 \leq \delta < \theta<1$, we have
\eqq{
\norm{G(t,\sigma)g}_{[L^{2}(0,1),\dd]_{\theta}} \leq \frac{c}{(t-\sigma)^{\theta-\delta}}\norm{g}_{[L^{2}(0,1),\dd]_{\delta}}
}{r4}
and  for $\theta \in (0,1), \delta \in (0,1]$, we have
\eqq{
\norm{A_{2}(t)G(t,\sigma)g}_{[L^{2}(0,1),\dd]_{\theta}} \leq \frac{c}{(t-\sigma)^{1+\theta-\delta}}\norm{g}_{[L^{2}(0,1),\dd]_{\delta}}.
}{r2}
Finally, for every $a \in (0,1)$ and every $0\leq \sigma < r < t \leq T$
\[
\norm{A_{2}(t)G(t,\sigma)g-A_{2}(r)G(r,\sigma)g}_{L^{2}(0,1)}
\]
\eqq{
\leq c\left(\frac{(t-r)^{a}}{(r-\sigma)^{1-\delta}}+\frac{1}{(r-\sigma)^{1-\delta}}-\frac{1}{(t-\sigma)^{1-\delta}}\right)
\norm{g}_{[L^{2}(0,1),\dd]_{\delta}}.
}{r3}

\no We would like to find a mild solution to (\ref{niecv}). For this purpose we rewrite the equation~(\ref{niecv}) in the integral form
\eqq{
v(x,t) = G(t,0)v_{0}(x) + \int_{0}^{t}G(t,\sigma)\frac{\dot{s}(\sigma)}{s(\sigma)}xv_{x}(x,\sigma)d\sigma.
}{wzorv}
We say that $v \in C([0,T];\dd)$ is a mild solution to (\ref{niecv}) if it satisfies (\ref{wzorv}).
\begin{theorem}\label{first}
  Let us assume that $v_{0} \in \dd$. Then, there exists a unique solution to~(\ref{wzorv}) belonging to $C([0,T];\dd)$.
\end{theorem}
\begin{proof}
  We will prove this result by the Banach fixed point theorem. We define the operator
  \[
  (Pv)(x,t) =  G(t,0)v_{0}(x) + \int_{0}^{t}G(t,\sigma)\frac{\dot{s}(\sigma)}{s(\sigma)}xv_{x}(x,\sigma)d\sigma.
  \]
  We will show that $P:C([0,T];\dd)\rightarrow C([0,T];\dd)$. Indeed, let $v \in C([0,T];\dd)$. Since $v_{0} \in \dd$ by \cite[Corollary 6.1.6. (iii)]{Lunardi} we obtain that $G(t,0)v_{0}~\in C([0,T];\dd)$.
  Let us pass to the second term. We will show that
\eqq{
A_{2}(t) \int_{0}^{t}G(t,\sigma)\frac{\dot{s}(\sigma)}{s(\sigma)}xv_{x}(x,\sigma)d\sigma \in C([0,T];L^{2}(0,1)).
}{r5}
Let us denote $\delta =\frac{\al}{\al+1}$ in the case $\al < \frac{1}{2}$ and for $\al \in [\frac{1}{2},1)$ let us mean by $\delta$ an arbitrary number belonging to the interval $(0,\frac{1}{2(1+\al)})$. We note that, since $v \in C([0,T];\dd)$, we have $v_{x} \in C([0,T];{}_{0}H^{\al}(0,1))$, hence $xv_{x} \in C([0,T];{}_{0}H^{\al}(0,1))$. Thus, we obtain that
\eqq{
  xv_{x} \in C([0,T];[L^{2}(0,1),\dd]_{\delta}).
}{vxm}
Let us denote
\eqq{
f(x,\sigma):=\frac{\dot{s}(\sigma)}{s(\sigma)}xv_{x}(x,\sigma).
}{deff}
Since for every $t \in [0,T]$ the operator $A_{2}(t)$ is closed by \cite[Proposition C.4]{Engel} we may pass with $A_{2}(t)$ under the integral sign. Hence, for any $0\leq \tau<t \leq T$ we may estimate as follows
\[
\norm{A_{2}(t)\izt G(t,\sigma)f(\cdot,\sigma)d\sigma -A_{2}(\tau)\int_{0}^{\tau} G(\tau,\sigma)f(\cdot,\sigma)d\sigma}_{L^{2}(0,1)}
\]
\[
= \norm{\izt A_{2}(t) G(t,\sigma)f(\cdot,\sigma)d\sigma -\int_{0}^{\tau}A_{2}(\tau) G(\tau,\sigma)f(\cdot,\sigma)d\sigma}_{L^{2}(0,1)} \leq
\]
\[
\int_{\tau}^{t}\norm{A_{2}(t)G(t,\sigma)f(\cdot,\sigma)}_{L^{2}(0,1)}d\sigma + \int_{0}^{\tau}\norm{(A_{2}(t)G(t,\sigma)-A_{2}(\tau)G(\tau,\sigma))f(\cdot,\sigma)}_{L^{2}(0,1)}d\sigma
\equiv: J_{1}.
\]
By (\ref{r1}) and (\ref{r3}) we may estimate $J_{1}$ further,
\[
J_{1}\leq c \norm{f}_{L^{\infty}(0,T;[L^{2}(0,1),\dd]_{\delta})}\int_{\tau}^{t} (t-\sigma)^{\delta-1}d\sigma
\]
\[
 + c \norm{f}_{L^{\infty}(0,T;[L^{2}(0,1),\dd]_{\delta})}\int_{0}^{\tau}\frac{(t-\tau)^{a}}{(\tau-\sigma)^{1-\delta}} + \frac{1}{(\tau-\sigma)^{1-\delta}}-\frac{1}{(t-\sigma)^{1-\delta}} d\sigma
\]
\[
\leq \frac{c}{\delta}\norm{f}_{L^{\infty}(0,T;[L^{2}(0,1),\dd]_{\delta})}
 \left(2(t-\tau)^{\delta} + (t-\tau)^{a}\tau^{\delta} + \tau^{\delta}-t^{\delta}\right)
\]
for any $a \in (0,1)$. The expression above tends to zero as $\tau\rightarrow t$ for any $0\leq \tau<t \leq T$, hence (\ref{r5}) is proven. We note that by (\ref{zals}) inclusion (\ref{r5}) leads to
\eqq{
\int_{0}^{t}G(t,\sigma)\frac{\dot{s}(\sigma)}{s(\sigma)}xv_{x}(x,\sigma)d\sigma \in C([0,T];\dd).
}{r6}
Thus, we have shown that $P:C([0,T];\dd)\rightarrow C([0,T];\dd)$.

It remains to show that $P$ is a contraction on $C([0,T_{1}];\dd)$ for $T_{1}$ small enough. To this end we fix $v, w \in C([0,T_{1}];\dd)$. Then, we may estimate using (\ref{zals}), (\ref{r1}) and~(\ref{vxm})
  \[
  \norm{Pv - Pw}_{C([0,T_{1}];\dd)} \leq \sup_{t\in (0,T_{1})}\frac{M}{b}\int_{0}^{t} \norm{G(t,\sigma)x[v_{x}-w_{x}](\cdot,\sigma)}_{\dd} d\sigma
  \]
  \[
   \leq \frac{cM}{b}\int_{0}^{T_{1}} (T_{1}-\sigma)^{\delta-1}d\sigma \norm{v_{x}-w_{x}}_{C([0,T_{1}];[L^{2}(0,1);\dd]_{\delta})} \leq \frac{cM}{b}\frac{T_{1}^{\delta}}{\delta}\norm{v-w}_{C([0,T_{1}];\dd)}.
  \]
   Hence, for $T_{1} < \left(\frac{b\delta}{cM}\right)^{\frac{1}{\delta}}$ the operator $P$ is a contraction on $C([0,T_{1}];\dd)$.
  Thus, by the Banach fixed point theorem, we obtain the existence of a unique solution to (\ref{wzorv}) on the interval $(0,T_{1}]$ which belongs to $C([0,T_{1}];\dd)$. Since the time interval may be estimated from below by a universal constant, after a finite number of steps we may extend the solution to $(0,T]$.
\end{proof}
  \begin{lem}\label{second}
The mild solution $v$ obtained in Theorem \ref{first} is a strong solution, i.e.  $v \in C([0,T];\dd)$, $v_{t} \in L^{\infty}(0,T;L^{2}(0,1))$ and it satisfies
  \[
  v_{t} - A(t)v = 0
  \]
  for almost all $t\in [0,T]$ in the sense of $L^{2}(0,1)$.
  \end{lem}
  \begin{proof}
  Using definition (\ref{deff}) we may rewrite (\ref{wzorv}) as follows
  \eqq{
  v(x,t) = G(t,0)v_{0}(x) + \izt G(t,\sigma)f(x,\sigma)d\sigma.
  }{vcf}
  The proof is based on the reasoning carried in the proof of \cite[Lemma 6.2.1]{Lunardi}. We note that, since $f$ is not H\"older continuous, we can not apply \cite[Lemma 6.2.1]{Lunardi} directly. We will show that $v$ which satisfies (\ref{vcf}) is differentiable.
  By the properties of an evolution operator (\cite[Corollary 6.1.6. (iii)]{Lunardi} ) for every $t \in [0,T]$ we have
  \[
  \frac{\p}{\p t}G(t,0)v_{0} = A_{2}(t) G(t,0)v_{0} \m{ in } L^{2}(0,1).
  \]
  We will calculate the difference quotient of the second term on the right hand side of (\ref{vcf}). Let us assume that $h > 0$, in the case $h<0$ the proof is similar. We have
  \[
  \frac{1}{h}\left[\int_{0}^{t+h} G(t+h,\sigma)f(x,\sigma)d\sigma - \izt G(t,\sigma)f(x,\sigma)d\sigma\right]
  \]
  \[
  = \frac{1}{h}\int_{0}^{t} (G(t+h,\sigma)-G(t,\sigma))f(x,\sigma)d\sigma
  +\frac{1}{h}\int_{t}^{t+h} G(t+h,\sigma)f(x,\sigma)d\sigma =: I_{1} + I_{2}.
  \]
  In order to deal with $I_{1}$ we recall that for every $0 \leq \sigma < t \leq T$ and every $g \in L^{2}(0,1)$ the following limit holds in $L^{2}(0,1)$
  \[
  \lim_{h\rightarrow 0}\frac{1}{h}(G(t+h,\sigma)-G(t,\sigma))g = A_{2}(t)G(t,\sigma)g.
  \]
    Making use of (\ref{vxm}) we obtain that $f \in L^{\infty}(0,T;[L^{2}(0,1),\dd]_{\delta})$, where $\delta = \frac{\al}{1+\al}$ for $\al \in (0,\frac{1}{2})$ and $\delta$ denotes any fixed number from the interval $(0,\frac{1}{2(1+\al)})$ if $\al \in [\frac{1}{2},1)$. Further, we note that
  \[
  \norm{\frac{1}{h}[G(t+h,\sigma) - G(t,\sigma)]f(\cdot,\sigma)}_{L^{2}(0,1)} =
  \norm{\frac{1}{h}\int_{t}^{t+h}\frac{\p}{\p p}G(p,\sigma)f(\cdot,\sigma)d\sigma}_{L^{2}(0,1)}
  \]
  \[
  =\norm{\frac{1}{h}\int_{t}^{t+h}A(p)G(p,\sigma)f(\cdot,\sigma)dp}_{L^{2}(0,1)}
  \leq \frac{c}{h}\int_{t}^{t+h}(p-\sigma)^{\delta-1}dp\norm{f}_{L^{\infty}(0,T;[L^{2}(0,1),\dd]_{\delta})}
  \]
  \[
  \leq c(t-\sigma)^{\delta-1}\norm{f}_{L^{\infty}(0,T;[L^{2}(0,1),\dd]_{\delta})},
  \]
 where we used~(\ref{r1}).
  Hence, we may apply the Lebesgue dominated convergence theorem to pass to the limit under the integral sign in $I_{1}$ and we get
  \[
  \frac{1}{h}\int_{0}^{t} (G(t+h,\sigma)-G(t,\sigma))f(x,\sigma)d\sigma\rightarrow \int_{0}^{t} A_{2}(t)G(t,\sigma)f(x,\sigma)d\sigma.
  \]
  We decompose $I_{2}$ as follows
  \[
  \frac{1}{h}\int_{t}^{t+h} G(t+h,\sigma)f(x,\sigma)d\sigma= \frac{1}{h}\int_{t}^{t+h} G(t,\sigma)f(x,\sigma)d\sigma
  \]\[
   +\frac{1}{h}\int_{t}^{t+h} (G(t+h,\sigma)-G(t,\sigma))f(x,\sigma)d\sigma = I_{2,1} + I_{2,2}.
  \]
We note that due to the Lebesgue differentiation theorem  in Banach spaces (see \cite{Bochner}) we obtain that $I_{2,1}$ converges to $f(x,t)$ in $L^{2}(0,1)$ for almost all $t \in (0,T]$. For $I_{2,2}$ we have
\[
I_{2,2} = \frac{1}{h}\int_{t}^{t+h} (G(t+h,t)-Id)G(t,\sigma)f(x,\sigma)d\sigma =  \frac{(G(t+h,t)-Id)}{h}\int_{t}^{t+h}G(t,\sigma)f(x,\sigma)d\sigma.
\]
Thus, using again the Lebesgue differentiation theorem in Banach spaces and the continuity of $G(t,\cdot)$ in $L^{2}(0,1)$ we obtain that $I_{2,2}$ converges to zero in $L^{2}(0,1)$ for almost all $t \in (0,T]$.
Summing up the results we obtain that the following identity holds in $L^{2}(0,1)$ for almost all $t \in [0,T]$
\[
v_{t}(x,t) = A_{2}(t)G(t,0)v_{0}(x) + A_{2}(t)\izt G(t,\sigma)f(x,\sigma)d\sigma + f(x,t).
\]
Applying formula (\ref{vcf}) and the definitions of $f$ and $A_{2}$ we obtain that
\[
v_{t}(x,t) = \frac{1}{s^{1+\al}(t)}\poch \da v(x,t) + \frac{\dot{s}(t)}{s(t)}xv_{x}(x,t)
\]
for almost all $t \in [0,T]$ in $L^{2}(0,1)$ and we obtain the claim of lemma.
  \end{proof}

\no Our aim is to obtain a solution to (\ref{niech}) regular enough to satisfy the weak extremum principle. As it will be seen in the final chapter, our solution $u$ has to fulfill the following:
there exists $\beta\in (\al,1)$ such that
\eqq{
  \m{for every } t\in (0,T) \m{ and every } 0<\ve<\omega < s(t) \hd  u(\cdot,t) \in W^{2,\frac{1}{1-\beta}}(\ve,\omega).
}{regumax}
Thus, we need to increase the space regularity of the transformed problem~(\ref{niecv}). The main difficulty is that, from what we have proved by now, $v_{x} \in~{}_{0}H^{\al}(0,1)$ but $v_{x}$ does not need to vanish at the right endpoint of the interval. Hence, we are allowed to consider $v_{x}$ as an element of the interpolation space $[L^{2}(0,1),\dd]_{\delta}$ only for $\delta$ smaller than $\frac{1}{2(1+\al)}$.
 However, in order to obtain higher regularity, we have to examine the behaviour of $A_{2}(t)G(t,\sigma)f(x,\sigma)$ more carefully. The next lemma establishes the regularity properties of an evolution operator $G(t,\sigma)$ acting on the elements of $H^{a}(0,1)$ for $a > \frac{1}{2}$.
 At first we will discuss the case $\al \in (\frac{1}{2},1)$. Then, we will present more technical result in the case $\al \in (0,\frac{1}{2}]$.
\begin{lem}\label{lemusi}
Let us assume that $\al \in (\frac{1}{2},1)$ and $u_{\sigma} \in {}_{0}H^{\al}(0,1)$. We denote by $u$ the solution to the equation
\begin{equation}\label{pol}
 \left\{ \begin{array}{ll}
u_{t} =A_{2}(t)u & \textrm{ for } 0<x<1, \hd 0\leq \sigma<t<T,\\
u(x,\sigma) = u_{\sigma}(x) & \textrm{ for } 0<x<1, \\
\end{array} \right. \end{equation}
given by the evolution operator generated by the family $A_{2}(t)$. Then, for every \hd $0~<\gamma<\al$, for every $0<\ve<\omega<1$ there exists a positive constant $c=c(\al,b,M,T,\ve,\omega,\gamma)$, where $b,M$ comes from (\ref{zals}), such that
\[
\norm{A_{2}(t)u(\cdot,t)}_{H^{\gamma}(\ve,\omega)} \leq c(t-\sigma)^{-\frac{1+\gamma}{1+\al}}\norm{u_{\sigma}}_{{}_{0}H^{\al}(0,1)}.
\]
\end{lem}
\begin{proof}
We note that $u_{\sigma} \in [L^{2}(0,1),\dd]_{\nu}$ for every $\nu \in (0,\frac{1}{2(1+\al)})$. Hence, by the theory of evolution operators $u \in C([\sigma,T];L^{2}(0,1))\cap  C((\sigma,T];\dd)\cap C^{1}((\sigma,T];L^{2}(0,1))$ and
\eqq{\norm{u(\cdot,t)}_{\dd} \leq c (t-\sigma)^{\nu-1}\norm{u_{\sigma}}_{[L^{2}(0,1),\dd]_{\nu}}.}{nowea}
We recall that the interpolation constant $c$ depends on the parameters of interpolation as well as from $\al, T$ and $b,M$ from (\ref{zals}). However, here and henceforth we neglect this dependency in notation and leave it just in the final results.
We fix $0<\ve<\omega<1$ and we set $\omega_{*}= \frac{1+\omega}{2}$. Let us discuss the approximate problem. We choose a sequence $\{\vf^{k}\}$ such that
\eqq{ \{\vf^{k}\} \subseteq \dd, \hd \vf^{k}\rightarrow u_{\sigma} \m{ in } {}_{0}H^{\al}(0,\omega_{*}) \m{ and }  \vf^{k}\rightarrow u_{\sigma} \m{ in } H^{\bar{\gamma}}(0,1) \m{ for every } \bar{\gamma} < \frac{1}{2}.}{noweb}
Then, applying (\ref{r1}) and (\ref{r2}), we obtain that the solution to
\begin{equation}\label{nowec}
 \left\{ \begin{array}{ll}
u^{k}_{t} =A_{2}(t)u^{k} & \textrm{ for } 0<x<1, \hd 0\leq \sigma<t<T,\\
u^{k}(x,\sigma) = \vf^{k}(x) & \textrm{ for } 0<x<1, \\
\end{array} \right. \end{equation}
satisfies for every $0\leq\bar{\gamma}<\bar{\gamma_{1}}<\frac{1}{2}$
\[
\norm{A_{2}(t)(u-u^{k})(\cdot,t)}_{H^{\bar{\gamma}}(0,1)} \leq c (t-\sigma)^{-1+\frac{(\bar{\gamma_{1}} - \bar{\gamma})}{1+\al}}\norm{u_{\sigma}-\vf^{k}}_{H^{\bar{\gamma_{1}}}(0,1)}.
\]
Hence, for every $0\leq\bar{\gamma}<\frac{1}{2}$
\eqq{
\poch \da u^{k} \rightarrow \poch \da u \m{ in } C_{loc}(\sigma,T;H^{\bar{\gamma}}(0,1)).
}{nowed}
Furthermore, for $k$ large enough and every $0\leq\bar{\gamma}<\bar{\gamma_{1}}<\frac{1}{2}$ we have
\eqq{
\norm{A_{2}(t) u^{k}(\cdot,t)}_{H^{\bar{\gamma}}(0,1)} \leq c (t-\sigma)^{-1+\frac{(\bar{\gamma_{1}} - \bar{\gamma})}{1+\al}}\norm{u_{\sigma}}_{H^{\bar{\gamma_{1}}}(0,1)}.
}{nowef}
We will prove the uniform bound of the sequence $\{u^{k}\}$ in more regular spaces locally on $(0,1)$.
To this end, we introduce a smooth nonnegative cut-off function $\eta$ such that $\eta(0)=\eta'(0)=0$, $\eta\equiv 1$ on~$[\ve,\omega]$, $\eta\equiv 0$ on $[\omega_{*},1]$.
We apply to (\ref{nowec}) the operator $\da$ and we multiply the obtained identity by $\eta$.

We note that for any smooth function $g$ and absolutely continuous f we have
\eqq{
g(x) \da f(x) = -\frac{\al}{\Gamma(1-\al)} \izx (x-p)^{-\al-1}(g(x)-g(p))f(p)dp + \da (fg)(x).
}{noweg}
Hence, we may write
\[
\eta \da \poch \da u^{k} = -\frac{\al}{\Gamma(1-\al)} \izx (x-p)^{-\al-1}(\eta(x)-\eta(p))\frac{\p}{\p p}\da u^{k}dp + \da (\poch\da u^{k} \cdot \eta )
\]
\[
= -\frac{\al}{\Gamma(1-\al)} \izx (x-p)^{-\al-1}(\eta(x)-\eta(p))\frac{\p}{\p p}\da u^{k}dp + \p^{\al}\poch(\eta \da u^{k}) - \p^{\al}(\eta'\da u^{k}),
\]
where we used $\eta(0) = 0$.
Hence, if we apply to (\ref{nowec}) the operator $\da$ and multiply the result by $\eta$ we get that
\begin{equation}\label{noweh}
 \left\{ \begin{array}{ll}
(\da u^{k}\cdot \eta)_{t} - A_{2}(t) (\da u^{k}\cdot\eta) =F^{k} & \textrm{ for } 0<x<1, \hd 0\leq \sigma<t<T,\\
(\da u^{k} \cdot \eta)(\cdot,\sigma) = \da\vf^{k}\cdot \eta & \textrm{ for } 0<x<1, \\
\end{array} \right. \end{equation}

where
\[
F^{k}:= -\frac{1}{s^{1+\al}(t)}\frac{\al}{\Gamma(1-\al)} \izx (x-p)^{-\al-1}(\eta(x)-\eta(p))\frac{\p}{\p p}\da u^{k}dp -\frac{1}{s^{1+\al}(t)} \p^{\al}(\eta'\da u^{k}).
\]
Let us show the uniform estimate on $L^{2}$ - norm of $F^{k}$.
We note that for every $x, p \in [0,1]$
\[
\abs{\frac{\eta(x)-\eta(p)}{x-p}} \leq \norm{\eta}_{W^{1,\infty}(0,1)},
\]
hence
\[
\frac{1}{\Gamma(1-\al)}\abs{\izx (x-p)^{-\al-1}(\eta(x)-\eta(p))A_{2}(t) u^{k}dp} \leq \norm{\eta}_{W^{1,\infty}(0,1)} I^{1-\al}\abs{A_{2}(t) u^{k}}.
\]
Since $I^{1-\al}$ is bounded on $L^{2}(0,1)$ we obtain that
\[
\norm{\frac{\al}{\Gamma(1-\al)}\izx (x-p)^{-\al-1}(\eta(x)-\eta(p))A_{2}(t) u^{k}dp }_{L^{2}(0,1)} \leq c(\ve,\omega)\norm{A_{2}(t) u^{k}}_{L^{2}(0,1)}.
\]
By Proposition \ref{eq} we may write
\[
\frac{1}{s^{1+\al}(t)}\norm{\p^{\al}(\eta'\da u^{k})}_{L^{2}(0,1)} \leq \frac{\ca}{s^{1+\al}(t)} \norm{(\eta'\da u^{k})}_{{}_{0}H^{\al}(0,1)} \leq c(\ve,\omega)\norm{A_{2}(t) u^{k}}_{L^{2}(0,1)}.
\]
Combining the last two estimates and (\ref{nowef}) we obtain that for every $0<\bar{\gamma} < \frac{1}{2}$
\eqq{
\norm{F^{k}(\cdot,t)}_{L^{2}(0,1)} \leq c(\ve,\omega)(t-\sigma)^{\frac{\bar{\gamma}}{(1+\al)}-1}\norm{u_{\sigma}}_{H^{\bar{\gamma}}(0,1)}.
}{nowej}
Moreover $F^{k} \in C((\sigma,T];L^{2}(0,1))\cap L^{1}(\sigma,T;L^{2}(0,1))$. Since $\eta(0)=\eta(1)=0$, by regularity of $u^{k}$ we obtain that $\da u^{k} \cdot \eta \in C([\sigma,T];L^{2}(0,1))\cap C((\sigma,T];\dd) \cap C^{1}((\sigma,T];L^{2}(0,1))$. Hence, by \cite[Corollary 6.2.4] {Lunardi} $\da u^{k} \cdot \eta$ satisfies the integral equality
\eqq{
(\da u^{k} \cdot \eta)(x,t) = G(t,\sigma)(\da \vf^{k} \cdot \eta)(x) + \int_{\sigma}^{t}G(t,\tau)F^{k}(x,\tau)d\tau.
}{nowei}
We fix $\gamma \in (0,1+\al)$, then
\[
\norm{(\da u^{k} \cdot \eta)(\cdot,t)}_{[L^{2},\dd]_{\frac{\gamma}{1+\al}}}\hspace{-0.4 cm} \leq \norm{G(t,\sigma)(\da \vf^{k} \cdot \eta)}_{[L^{2},\dd]_{\frac{\gamma}{1+\al}}} \hspace{-0.4 cm} + \int_{\sigma}^{t}\norm{G(t,\tau)F^{k}(\cdot,\tau)}_{[L^{2},\dd]_{\frac{\gamma}{1+\al}}} \hspace{-0.2 cm} d\tau
\]
and we apply estimate (\ref{r4}) to obtain
\[
\norm{(\da u^{k} \cdot \eta)(\cdot,t)}_{H^{\gamma}(0,1)} \leq c(t-\sigma)^{-\frac{\gamma}{1+\al}}\norm{\da \vf^{k} \cdot \eta}_{L^{2}(0,1)} + c \int_{\sigma}^{t}(t-\tau)^{-\frac{\gamma}{1+\al}}\norm{F^{k}(\cdot,\tau)}_{L^{2}(0,1)}d\tau.
\]
From (\ref{nowej}) we infer that for every $0<\bar{\gamma}<\frac{1}{2}$ there holds
\[
\norm{(\da u^{k} \cdot \eta)(\cdot,t)}_{H^{\gamma}(0,1)}
\]
\[
 \leq c(t-\sigma)^{-\frac{\gamma}{1+\al}}\norm{\p^{\al} \vf^{k}}_{L^{2}(0,\omega_{*})} +c(\ve,\omega) \int_{\sigma}^{t}(t-\tau)^{-\frac{\gamma}{1+\al}}(\tau-\sigma)^{\frac{\bar{\gamma}}{1+\al}-1}d\tau
\norm{u_{\sigma}}_{H^{\bar{\gamma}}(0,1)}.
\]
By (\ref{noweb}) we get that
\[
\norm{(\da u^{k} \cdot \eta)(\cdot,t)}_{H^{\gamma}(0,1)} \leq c(\ve,\omega)[(t-\sigma)^{-\frac{\gamma}{1+\al}}\norm{u_{\sigma}}_{{}_{0}H^{\al}(0,\omega_{*})} +(t-\sigma)^{\frac{\bar{\gamma}-\gamma}{1+\al}}\norm{u_{\sigma}}_{{}_{0}H^{\bar{\gamma}}(0,1)}].
\]
Since $\gamma$ is an arbitrary number from the interval $(0,1+\al)$, the estimate above implies the following: for every $\gamma_{1} < \al$
\[
\norm{(\poch \da u^{k} \cdot \eta)(\cdot,t)}_{H^{\gamma_{1}}(0,1)} \leq c(\ve,\omega)(t-\sigma)^{-\frac{\gamma_{1}+1}{1+\al}}\norm{u_{\sigma}}_{{}_{0}H^{\al}(0,1)}.
\]
Since we obtained a uniform estimate, in view of (\ref{nowed}) we get that, on the subsequence
\[
\poch \da u^{k}(\cdot,t)\rightharpoonup \poch \da u(\cdot,t) \m{ in } H^{\gamma_{1}}(\ve,\omega) \m{ for every} \hd t \in (\sigma,T).
\]
Furthermore, by weak lower semi-continuity of the norm, we arrive at the estimate: for every $0<\gamma_{1}<\al$, for every $t \in (\sigma,T)$
\[
\norm{\poch \da u(\cdot,t)}_{H^{\gamma_{1}}(\ve,\omega)} \leq c(t-\sigma)^{-\frac{\gamma_{1}+1}{1+\al}}\norm{u_{\sigma}}_{{}_{0}H^{\al}(0,1)},
\]
where $c=c(\al,b,M,\ve,\omega,T,\gamma_{1})$. This together with (\ref{zals}) finishes the proof.
\end{proof}
Now, we present a more technical result which is necessary to increase the regularity of solution in the case $0<\al \leq \frac{1}{2}$.
\begin{lem}\label{lemusi2}
Let $0<\al \leq \frac{1}{2}$. Let us assume that $u_{\sigma} \in H^{\beta}_{loc}(0,1) \cap H^{\bar{\gamma}}(0,1)$ for fixed $\frac{1}{2}<\beta<1$ and fixed $0<\bar{\gamma} < \frac{1}{2}$. We denote by $u$ the solution to the equation
\begin{equation}\label{polc}
 \left\{ \begin{array}{ll}
u_{t} =A_{2}(t)u & \textrm{ for } 0<x<1, \hd 0\leq \sigma<t<T,\\
u(x,\sigma) = u_{\sigma}(x) & \textrm{ for } 0<x<1, \\
\end{array} \right. \end{equation}
given by the evolution operator generated by the family $A_{2}(t)$. Then, for every $\max\{\beta-\al, \beta - \bar{\gamma}\}<\beta_{1}<\beta$, for every $0<\ve<\omega<1$, there exists a positive constant $c=c(\al,b,M,T,\ve,\omega,\beta,\beta_{1})$, such that there holds
\[
\norm{A_{2}(t)u(\cdot,t)}_{H^{\beta_{1}}(\ve,\omega)} \leq c(t-\sigma)^{-\frac{\beta_{1}-\beta+\al+1}{1+\al}}(\norm{u_{\sigma}}_{H^{\beta}(\frac{\ve}{2},\frac{1+\omega}{2})}+\norm{u_{\sigma}}_{H^{\bar{\gamma}}(0,1)}).
\]
\end{lem}
\begin{proof}
We will modify the proof of Lemma \ref{lemusi}. At first we fix $0<\ve<\omega<1$ and we set $\omega_{*}=\frac{1+\omega}{2}$.
We choose a sequence $\{\vf^{k}\} \subseteq \dd$ such that
\eqq{\vf^{k}(0)=0, \hd \vf^{k}\rightarrow u_{\sigma} \m{ in } H^{\beta}(\frac{\ve}{2},\omega_{*}) \m{ and }\hd \vf^{k}\rightarrow u_{\sigma} \m{ in } H^{\bar{\gamma}}(0,1).}{nowebc}
As in the proof of Lemma \ref{lemusi}, we obtain that the solution to
\begin{equation}\label{nowecc}
 \left\{ \begin{array}{ll}
u^{k}_{t} =A_{2}(t)u^{k} & \textrm{ for } 0<x<1, \hd 0\leq \sigma<t<T,\\
u^{k}(x,\sigma) = \vf^{k}(x) & \textrm{ for } 0<x<1, \\
\end{array} \right. \end{equation}
satisfies for every $0\leq \bar{\gamma_{1}}<\bar{\gamma}$
\[
\norm{A_{2}(t)(u-u^{k})(\cdot,t)}_{H^{\bar{\gamma_{1}}}(0,1)} \leq c (t-\sigma)^{-1+\frac{\bar{\gamma} - \bar{\gamma_{1}}}{1+\al}}\norm{u_{\sigma}-\vf^{k}}_{H^{\bar{\gamma}}(0,1)}.
\]
Hence, for every $0\leq\bar{\gamma_{1}}<\bar{\gamma}$ there holds the following limit
\eqq{
\poch \da u^{k} \rightarrow \poch \da u \m{ in } C_{loc}(\sigma,T;H^{\bar{\gamma_{1}}}(0,1)).
}{nowedc}
Moreover, for $k$ large enough and every $0\leq\bar{\gamma_{1}}<\bar{\gamma}$ we get
\eqq{
\norm{A_{2}(t) u^{k}(\cdot,t)}_{H^{\bar{\gamma_{1}}}(0,1)} \leq c (t-\sigma)^{-1+\frac{\bar{\gamma} - \bar{\gamma_{1}}}{1+\al}}\norm{u_{\sigma}}_{H^{\bar{\gamma}}(0,1)}.
}{nowefc}
We introduce a smooth nonnegative function $\eta\equiv 0$ on $[0,\ve/2]$, $\eta\equiv 1$ on $[\ve,\omega]$, $\eta\equiv 0$ on $[\omega_{*},1]$.
We apply to (\ref{nowecc}) the operator $D^{\beta}$ and we multiply the result by $\eta$.
Making use of identity (\ref{noweg}) we may calculate as follows
\[
\eta D^{\beta} \poch \da u^{k} = -\frac{\beta}{\Gamma(1-\beta)} \izx (x-p)^{-\beta-1}(\eta(x)-\eta(p))\frac{\p}{\p p}\da u^{k}dp + D^{\beta} (\poch\da u^{k} \cdot \eta ).
\]
Using $\eta(0)=0$ in the first identity and $(D^{\beta}u^{k})(0,t) = 0$ in the second one, we have
\[
D^{\beta} (\poch\da u^{k} \cdot \eta ) = \p^{\beta} (\p^{\al} \poch u^{k} \cdot \eta ) =  \p^{\beta} (\p^{\al} D^{1-\beta}D^{\beta} u^{k} \cdot \eta ) =\p^{\beta} ( D^{1+\al-\beta}D^{\beta} u^{k} \cdot \eta ).
\]
We apply again (\ref{noweg}) to get
\[
 D^{1+\al-\beta}D^{\beta} u^{k} \cdot \eta  = \frac{\beta-\al-1}{\Gamma(\beta-\al)}\izx (x-p)^{\beta-\al-2}(\eta(x)-\eta(p))D^{\beta} u^{k}dp + D^{1+\al-\beta}(\eta D^{\beta}u^{k}).
\]
Finally, we note that
\[
\p^{\beta}D^{1+\al-\beta}(\eta D^{\beta}u^{k}) = \poch I^{1-\beta}I^{\beta-\al}\poch(\eta D^{\beta}u^{k}) = \poch I^{1-\al}\poch(\eta D^{\beta}u^{k}) = \p^{\al}\poch(\eta D^{\beta}u^{k}),
\]
where in the last equality we used the fact that $\eta\equiv 0$ on $[0,\ve/2]$.
Hence, if we apply to (\ref{nowecc}) the operator $D^{\beta}$ and multiply the obtained identity by $\eta$, we arrive at
\begin{equation}\label{nowehc}
 \left\{ \begin{array}{ll}
(D^{\beta} u^{k}\cdot \eta)_{t} - A_{2}(t) (D^{\beta} u^{k}\cdot\eta) =F^{k} & \textrm{ for } 0<x<1, \hd 0\leq \sigma<t<T,\\
(D^{\beta} u^{k} \cdot \eta)(\cdot,\sigma) = D^{\beta}\vf^{k}\cdot \eta & \textrm{ for } 0<x<1, \\
\end{array} \right. \end{equation}

where
\[
F^{k}:= -\frac{1}{s^{1+\al}(t)}\frac{\beta}{\Gamma(1-\beta)} \izx (x-p)^{-\beta-1}(\eta(x)-\eta(p))\frac{\p}{\p p}\da u^{k}dp
\]
\[
 +\frac{1}{s^{1+\al}(t)} \frac{\beta-\al-1}{\Gamma(\beta-\al)}\p^{\beta} \izx (x-p)^{\beta-\al-2}(\eta(x)-\eta(p))D^{\beta} u^{k}dp =:F^{k}_{1}+F^{k}_{2}.
\]
We will prove a uniform estimate of the $L^{2}$-norm of $F^{k}$. At first, we note that, as in the proof of (\ref{nowej}), by continuity of fractional integral in $L^{2}$ we obtain
\[
\norm{F^{k}_{1}(\cdot,t) }_{L^{2}(0,1)} \leq c(\ve,\omega)\norm{A_{2}(t) u^{k}}_{L^{2}(0,1)}.
\]
To estimate $F^{k}_{2}$ we note that
\[
\Gamma(1-\beta)\p^{\beta} \izx (x-p)^{\beta-\al-2}(\eta(x)-\eta(p))D^{\beta}u^{k}(p)dp
\]
\[
 =
\poch \izx (x-p)^{-\beta}\int_{0}^{p}(p-\tau)^{\beta-\al-2}(\eta(p) - \eta(\tau))D^{\beta}u^{k}(\tau)d\tau dp
\]
\[
= \poch \izx D^{\beta}u^{k}(\tau) \int_{\tau}^{x}(x-p)^{-\beta}(p-\tau)^{\beta-\al-2}(\eta(p) - \eta(\tau))dp d\tau = \podst{p=\tau+w(x-\tau)}{dp = (x-\tau)dw}
\]
\[
= \poch \izx D^{\beta-\al}D^{\al}u^{k}(\tau) (x-\tau)^{-\al-1} \izj (1-w)^{-\beta}w^{\beta-\al-2}(\eta(\tau+w(x-\tau)) - \eta(\tau))dw d\tau
\]
\[
= \poch \hspace{-0.1 cm}\izx \hspace{-0.2 cm}\int_{0}^{\tau}\hspace{-0.2 cm}\frac{(\tau-p)^{\al-\beta}}{\Gamma(1-\beta+\al)}\frac{\p}{\p p} D^{\al}u^{k}(p)dp (x-\tau)^{-\al-1}\hspace{-0.2 cm} \izj \frac{(\eta(\tau+w(x-\tau)) - \eta(\tau))}{(1-w)^{\beta}w^{2+\al-\beta}}dw d\tau.
\]
Applying the Fubini theorem and the substitution $\tau = p + a(x-p)$ we obtain further
\[
\Gamma(1-\beta)\p^{\beta} \izx (x-p)^{\beta-\al-2}(\eta(x)-\eta(p))D^{\beta}u^{k}(p)dp
\]
\[
= \frac{1}{\Gamma(1-\beta+\al)}\poch \izx \frac{\p}{\p p} D^{\al}u^{k}(p)(x-p)^{-\beta} H(x,p)dp,
\]
where
\[
H(x,p)\hspace{-0.1 cm} = \hspace{-0.2 cm}
\izj a^{\al-\beta}(1-a)^{-\al-1} \hspace{-0.2 cm} \izj \frac{\eta(p+a(x-p)+w(x-p)(1-a)) - \eta(p+a(x-p))}{(1-w)^{\beta}w^{2+\al-\beta}}dw da.
\]
We note that for every $a,w,p,x \in (0,1)$
\eqq{
\abs{\frac{(\eta(p+a(x-p)+w(x-p)(1-a)) - \eta(p+a(x-p)))}{(x-p)(1-a)w}} \leq \norm{\eta}_{W^{1,\infty}(0,1)}.
}{etap}
Hence,
\eqq{
\abs{H(x,p)} \leq  \norm{\eta}_{W^{1,\infty}(0,1)}\abs{x-p} B(1-\al,\al-\beta+1)B(\beta+1,\beta-\al)  \rightarrow 0 \m{ as } p\rightarrow x.
}{hza}
Thus, performing differentiation we arrive at the following identity
\[
\Gamma(1-\beta+\al)\Gamma(1-\beta)\p^{\beta} \izx (x-p)^{\beta-\al-2}(\eta(x)-\eta(p))D^{\beta}u^{k}(p)dp
\]
\[
=  \izx \frac{\p}{\p p} D^{\al}u^{k}(p)(x-p)^{-\beta} \poch H(x,p)dp-\beta  \izx \frac{\p}{\p p} D^{\al}u^{k}(p)(x-p)^{-\beta} \frac{H(x,p)}{x-p}dp.
\]
We note that for every $0\leq p < x \leq 1$
\[
\abs{ \poch H(x,p)} + \abs{\frac{H(x,p)}{x-p}} \leq c(\al,\beta,\ve,\omega).
\]
Indeed, we may show in a similar way as in (\ref{hza}) that for $p,x \in (0,1)$
\[
\abs{\poch H(x,p)} \leq c(\al,\beta)\norm{\eta}_{W^{2,\infty}(0,1)}.
\]
Finally, we arrive at
\[
\norm{F^{k}_{2}(\cdot,t)}_{L^{2}(0,1)}
\leq c(\al,\beta,\ve,\omega)\norm{I^{1-\beta}\abs{A_{2}(t) u^{k}}}_{L^{2}(0,1)} \leq c(\al,\beta,\ve,\omega)\norm{A_{2}(t) u^{k}}_{L^{2}(0,1)}.
\]
For clarity we neglect in notation the dependance in $c$ of other constants then $\omega,\ve$.
Hence, in view of (\ref{nowefc}) we obtain that
\eqq{
\norm{F^{k}(\cdot,t)}_{L^{2}(0,1)} \leq c(\ve,\omega)(t-\sigma)^{\frac{\bar{\gamma}}{(1+\al)}-1}\norm{u_{\sigma}}_{H^{\bar{\gamma}}(0,1)}.
}{nowejc}
Making use of the regularity of $u^{k}$ and the fact that $\eta(0)=\eta(1)=0$ we may apply \cite[Corollary 6.2.4] {Lunardi} to deduce that $D^{\beta} u^{k} \cdot \eta$ satisfies the integral identity
\eqq{
(D^{\beta} u^{k} \cdot \eta)(x,t) = G(t,\sigma)(D^{\beta} \vf^{k} \cdot \eta)(x) + \int_{\sigma}^{t}G(t,\tau)F^{k}(x,\tau)d\tau.
}{noweic}
We fix $\gamma \in (0,1+\al)$, then we may write
\[
\norm{(D^{\beta} u^{k} \cdot \eta)(\cdot,t)}_{[L^{2},\dd]_{\frac{\gamma}{1+\al}}} \leq \norm{G(t,\sigma)(D^{\beta} \vf^{k} \cdot \eta)}_{[L^{2},\dd]_{\frac{\gamma}{1+\al}}} + \int_{\sigma}^{t}\norm{G(t,\tau)F^{k}(\cdot,\tau)}_{[L^{2},\dd]_{\frac{\gamma}{1+\al}}}d\tau.
\]
Hence,
\[
\norm{(D^{\beta} u^{k} \cdot \eta)(\cdot,t)}_{H^{\gamma}(0,1)} \leq c(t-\sigma)^{-\frac{\gamma}{1+\al}}\norm{D^{\beta} \vf^{k} \cdot \eta}_{L^{2}(0,1)} +c \int_{\sigma}^{t}(t-\tau)^{-\frac{\gamma}{1+\al}}\norm{F^{k}(\cdot,\tau)}_{L^{2}(0,1)}d\tau,
\]
where we applied estimate (\ref{r4}) for the first term on the right hand side and then we did it again for the second one.
Applying (\ref{nowejc}) we obtain
\[
\norm{(D^{\beta} u^{k} \cdot \eta)(\cdot,t)}_{H^{\gamma}(0,1)}
\]
\[
\leq c(t-\sigma)^{-\frac{\gamma}{1+\al}}\norm{D^{\beta} \vf^{k}\cdot\eta}_{L^{2}(0,1)} +c \int_{\sigma}^{t}(t-\tau)^{-\frac{\gamma}{1+\al}}(\tau-\sigma)^{\frac{\bar{\gamma}}{1+\al}-1}d\tau
\norm{u_{\sigma}}_{H^{\bar{\gamma}}(0,1)}.
\]
We note that by the identity (\ref{noweg}) we have
\[
D^{\beta} \vf^{k}\cdot\eta = -\frac{\beta}{\Gamma(1-\beta)}\izx (x-p)^{-\beta}\frac{\eta(x)-\eta(p)}{x-p} \vf^{k}(p)dp + D^{\beta}(\vf^{k}\cdot \eta).
\]
Hence,
\[
\norm{D^{\beta} \vf^{k}\cdot\eta}_{L^{2}(0,1)} \leq c(\beta,\ve,\omega)\norm{\vf^{k}}_{L^{2}(0,1)} + \norm{\vf^{k}\eta}_{H^{\beta}(0,1)}
\]
and by (\ref{nowebc})
\[
\norm{D^{\beta} \vf^{k}\cdot\eta}_{L^{2}(0,1)} \leq c(\beta,\ve,\omega)(\norm{u_{\sigma}}_{H^{\beta}(\frac{\ve}{2},\omega_{*})}+\norm{u_{\sigma}}_{\ld}).
\]
Thus, we get that
\[
\norm{(D^{\beta} u^{k} \cdot \eta)(\cdot,t)}_{H^{\gamma}(0,1)} \leq c(t-\sigma)^{-\frac{\gamma}{1+\al}}(\norm{u_{\sigma}}_{H^{\beta}(\frac{\ve}{2},\omega_{*})}+\norm{u_{\sigma}}_{\ld}) + c(t-\sigma)^{\frac{\bar{\gamma}-\gamma}{1+\al}}\norm{u_{\sigma}}_{{}_{0}H^{\bar{\gamma}}(0,1)}.
\]
Since $\gamma$ is an arbitrary number from the interval $(0,1+\al)$, the estimate above implies the following: for every $\gamma_{1} < \al$
\eqq{
\norm{(\poch D^{\beta} u^{k} \cdot \eta)(\cdot,t)}_{H^{\gamma_{1}}(0,1)} \leq c(\ve,\omega)(t-\sigma)^{-\frac{\gamma_{1}+1}{1+\al}}
(\norm{u_{\sigma}}_{H^{\beta}(\frac{\ve}{2},\omega_{*})}+\norm{u_{\sigma}}_{{}_{0}H^{\bar{\gamma}}(0,1)}).
}{nowekd}
We will show that this leads to
\eqq{
\norm{\eta\poch \da u^{k}(\cdot,t)}_{H^{\beta_{1}}(0,1)} \leq  c(t-\sigma)^{-\frac{\beta_{1}-\beta+\al+1}{1+\al}}
(\norm{u_{\sigma}}_{H^{\beta}(\frac{\ve}{2},\omega_{*})}+\norm{u_{\sigma}}_{{}_{0}H^{\bar{\gamma}}(0,1)} )
}{nowekc}
for every $\max\{\beta-\al,\beta-\bar{\gamma}\}<\beta_{1}<\beta$ and where $c=c(\al,b,M,T,\ve,\omega,\beta_{1},\beta)$.
Indeed, we have
\[
\poch D^{\beta} u^{k} = \poch D^{\beta-\al}\da u^{k} = \eta \p^{\beta-\al}\p^{1-(\beta-\al)}D^{\beta-\al}\da u^{k} =  \p^{\beta-\al} \poch \da u^{k}.
\]
Hence,
\eqq{
\eta\poch D^{\beta} u^{k} = \frac{\al-\beta}{\Gamma(1-(\al-\beta))}\izx (x-p)^{\al-\beta}\frac{\eta(x)-\eta(p)}{x-p}\poch \da u^{k}(p)dp + \p^{\beta-\al}(\eta \poch \da u^{k}).
}{noweke}
Let us estimate $H^{\gamma_{1}}$ - norm of the second term on the right hand side for $\gamma_{1}~<~\al$.
In view of Proposition \ref{eq}, it is enough to estimate the $L^{2}$- norm of $\p^{\gamma_{1}}\izx (x-p)^{\al-\beta-1}(\eta(x)-\eta(p))\poch \da u^{k}(p)dp$.
We note that
\[
\Gamma(1-\gamma_{1})\p^{\gamma_{1}} \izx (x-p)^{\al-\beta-1}(\eta(x)-\eta(p))\poch \da u^{k}(p)dp
\]
\[
 =
\poch \izx (x-p)^{-\gamma_{1}}\int_{0}^{p}(p-\tau)^{\al-\beta-1}(\eta(p) - \eta(\tau))\poch \da u^{k}(\tau)d\tau dp
\]
\[
= \poch \izx \poch \da u^{k}(\tau) \int_{\tau}^{x}(x-p)^{-\gamma_{1}}(p-\tau)^{\al-\beta-1}(\eta(p) - \eta(\tau))dp d\tau = \podst{p=\tau+w(x-\tau)}{dp = (x-\tau)dw}
\]
\[
= \poch \izx \poch \da u^{k}(\tau)(x-\tau)^{\al-\beta-\gamma_{1}} \izj (1-w)^{-\gamma_{1}}w^{\al-\beta-1}(\eta(\tau+w(x-\tau)) - \eta(\tau))dw d\tau
\]
We note that we may differentiate under the integral sign. Indeed,
\[
\abs{\poch \da u^{k}(\tau)(x-\tau)^{\al-\beta-\gamma_{1}} \izj (1-w)^{-\gamma_{1}}w^{\al-\beta-1}(\eta(\tau+w(x-\tau)) - \eta(\tau))dw}
\]
\[
 \leq \norm{\eta}_{W^{1,\infty}(0,1)} B(1-\gamma_{1},1+\al-\beta)\abs{\poch \da u^{k}(\tau) (x-\tau)^{1+\al-\beta-\gamma_{1}}}\rightarrow 0 \m{ as } \tau\rightarrow x^{-}.
\]
Thus, proceeding with differentiation, we have
\[
\Gamma(1-\gamma_{1})\p^{\gamma_{1}} \izx (x-p)^{\al-\beta-1}(\eta(x)-\eta(p))\poch \da u^{k}(p)dp
\]
\[
= \izx \poch \da u^{k}(\tau) (x-\tau)^{\al-\beta-\gamma_{1}} \izj (1-w)^{-\gamma_{1}}w^{\al-\beta}\eta'(\tau+w(x-\tau))dw d\tau
\]
\[
+(\al-\beta-\gamma_{1})\izx \poch \da u^{k}(\tau) (x-\tau)^{\al-\beta-\gamma_{1}-1} \izj (1-w)^{-\gamma_{1}}w^{\al-\beta-1}(\eta(\tau+w(x-\tau)) - \eta(\tau))dw.
\]
Thus, estimating the $L^{2}$- norm of expression above we arrive at
\[
\norm{\izx (x-p)^{\al-\beta-1}(\eta(x)-\eta(p))\poch \da u^{k}(p)dp}_{H^{\gamma_{1}}(0,1)}
\]
\[
\leq c(\al,\beta,\gamma_{1})\norm{\eta}_{W^{1,\infty}(0,1)}\norm{I^{1+\al-\beta-\gamma_{1}}\abs{\poch \da u^{k}}}_{L^{2}(0,1)}\leq c(\al,\beta,\gamma_{1},\ve,\omega)\norm{\poch \da u^{k}}_{L^{2}(0,1)},
\]
where in the last estimate we applied boundedness of fractional integral in $L^{2}$.
Making use of estimates (\ref{nowefc}) and (\ref{nowekd}) in identity (\ref{noweke}) we obtain that for any $\max\{0,\al-\bar{\gamma}\}<\gamma_{1} < \al$, there holds
\[
\norm{\p^{\beta-\al}(\eta \poch \da u^{k})}_{H^{\gamma_{1}}(0,1)} \leq c(t-\sigma)^{-\frac{\gamma_{1}+1}{1+\al}}
(\norm{u_{\sigma}}_{H^{\beta}(\frac{\ve}{2},\omega_{*})}+\norm{u_{\sigma}}_{{}_{0}H^{\bar{\gamma}}(0,1)}).
\]
Hence, for any $\max\{0,\al-\bar{\gamma}\}<\gamma_{1}<\al$ we have
\[
\norm{\p^{\gamma_{1}}D^{\beta-\al}(\eta \poch \da u^{k})}_{L^{2}(0,1)} \leq c(t-\sigma)^{-\frac{\gamma_{1}+1}{1+\al}}
(\norm{u_{\sigma}}_{H^{\beta}(\frac{\ve}{2},\omega_{*})}+\norm{u_{\sigma}}_{{}_{0}H^{\bar{\gamma}}(0,1)}).
\]
Taking $\gamma_{1} = \beta_{1}-\beta+\al$, where $\max\{\beta-\al,\beta-\bar{\gamma}\}<\beta_{1}<\beta$  and applying Proposition~\ref{eq} we arrive at (\ref{nowekc}). Estimate (\ref{nowekc}) together with weak lower semi-continuity of the norm  finishes the proof.
\end{proof}

After having established Lemma \ref{lemusi} and Lemma \ref{lemusi2}, we have to state one additional lemma.

\begin{lem}\label{local}
Let $f\in {}_{0}H^{\al}(0,1)$ for $\al \in (0,1)$ and $\p^{\al}f \in H^{\beta}_{loc}(0,1)$ for $\beta \in (\frac{1}{2},1]$. Then $f \in H^{\beta+\al}_{loc}(0,1)$ and for every $0<\delta <\omega<1$ there exists a positive constant $c=c(\al,\beta,\delta,\omega)$ such that
\eqq{
\norm{f}_{H^{\beta+\al}(\delta,\omega)} \leq c(\norm{f}_{{}_{0}H^{\al}(0,\omega)} + \norm{\p^{\al}f}_{H^{\beta}(\frac{\delta}{2},\omega)}).
}{nowelc}
\end{lem}
\begin{proof}
Let us fix $0<\delta <\omega<1$. Then, by the assumption we have $\p^{\al}f \in H^{\beta}(\frac{\delta}{2},\omega)$. Thus, we may write for $x > \delta/2$
\[
f(x) = I^{\al}(\p^{\al}f-\p^{\al}f(\delta/2))(x) + I^{\al}(\p^{\al}f(\delta/2))(x)
\]
\[
=\frac{1}{\Gamma(\al)}\int_{0}^{\frac{\delta}{2}}(x-p)^{\al-1}(\p^{\al}f(p)-\p^{\al}f(\delta/2))dp+
\frac{1}{\Gamma(\al)}\int_{\frac{\delta}{2}}^{x}(x-p)^{\al-1}(\p^{\al}f(p)-\p^{\al}f(\delta/2))dp
\]
\[
+\frac{1}{\Gamma(\al)}\p^{\al}f(\delta/2)\int_{0}^{\frac{\delta}{2}}(x-p)^{\al-1}dp+
\frac{1}{\Gamma(\al)}\p^{\al}f(\delta/2)\int_{\frac{\delta}{2}}^{x}(x-p)^{\al-1}dp
\]
\[
=\frac{1}{\Gamma(\al)}\int_{0}^{\frac{\delta}{2}}(x-p)^{\al-1}\p^{\al}f(p)dp + I^{\al}_{\frac{\delta}{2}}(\p^{\al}f-\p^{\al}f(\delta/2))(x) + \frac{1}{\Gamma(1+\al)}(x-\delta/2)^{\al}\p^{\al}f(\delta/2).
\]
The third component belongs to $H^{\al+\beta}(\delta,\omega)$, while, by Corollary \ref{eqc}, the second one belongs to ${}_{0}H^{\al+\beta}(\frac{\delta}{2},\omega)$. We may show that the first component belongs to $H^{2}(\delta,\omega)$. Indeed, we have
\[
\frac{d^{2}}{dx^{2}}\int_{0}^{\frac{\delta}{2}}(x-p)^{\al-1}\p^{\al}f(p)dp = (\al-1)(\al-2)\int_{0}^{\frac{\delta}{2}}(x-p)^{\al-3}\p^{\al}f(p)dp
\]
and
\[
\int_{\delta}^{\omega}\abs{\int_{0}^{\frac{\delta}{2}}(x-p)^{\al-3}\p^{\al}f(p)dp}^{2}dx \leq \int_{\delta}^{\omega}(x-\frac{\delta}{2})^{2(\al-3)}\int_{0}^{\frac{\delta}{2}}\abs{\p^{\al}f(p)}^{2}dpdx
\]
\[
 \leq (\delta/2)^{2(\al-3)}\norm{\p^{\al}f}_{L^{2}(0,\omega)}^{2}.
\]
Thus,  $f \in H^{\al+\beta}(\delta,\omega)$ and the estimate (\ref{nowelc}) follows by the Sobolev embedding and Corollary \ref{eqc}.
\end{proof}

\no Finally, we are able to improve the space regularity of solutions to~(\ref{niecv}). We decompose interval $(0,1)$ as follows
\[
(0,1) = \left(\bigcup_{k=1}^{\infty}\left(\frac{1}{k+1},\frac{1}{k}\right]\right) \setminus \{1\}.
\]
Then, for each $\al \in (0,1)$ we may choose $k \in \mathbb{N} \setminus \{0\}$ such that $\al \in (\frac{1}{k+1},\frac{1}{k}]$.
 We will discuss separately the case for each $k \in \mathbb{N} \setminus \{0\}$. The proof for $\al \in (\frac{1}{2},1)$ requires just one step, however for $\al \in (\frac{1}{k+1},\frac{1}{k}]$ we need to repeat the reasoning $k$ times.
 \begin{lem}\label{third}
Let us assume that $v_{0} \in \dd$, $\al \in (0,1)$. We choose $k \in \mathbb{N} \setminus \{0\}$ such that  $\al \in (\frac{1}{k+1},\frac{1}{k}]$. Then, for every $\gamma_{k} \in (\al,(k+1)\al)$  the solution to (\ref{niecv}) obtained in Theorem \ref{first} satisfies
 \eqq{
 v \in L^{\infty}_{loc}(0,T;H^{\gamma_{k}+1}_{loc}(0,1)) \m{ and } \p^{\al}v_{x} \in L^{\infty}_{loc}(0,T;H^{\gamma_{k}-\al}_{loc}(0,1)).
 }{atreg}
 Furthermore,
 \[
 v \in L^{\frac{1+\al}{k\al}}(0,T;H^{\gamma_{k}+1}_{loc}(0,1)) \m{ and } \p^{\al}v_{x} \in L^{\frac{1+\al}{k\al}}(0,T;H^{\gamma_{k}-\al}_{loc}(0,1)).
 \]
 \end{lem}
 \begin{proof}

Let us denote $\delta=\frac{\al}{\al+1}$ in the case $\al \neq \frac{1}{2}$ and for $\al = \frac{1}{2}$ by $\delta$ we mean any number from the interval $(0,\frac{1}{3})$. We fix $0<\ve<\omega<1$.
We apply to (\ref{vcf}) the operator $A_{2}(t)$ and estimate its norm in the interpolation space. Firstly, we consider the case $\al \in  (\frac{1}{2},1)$. In this case, we have $f \in L^{\infty}(0,T;{}_{0}H^{(1+\al)\delta}(0,1))$. Thus, by Lemma \ref{lemusi} we obtain for any $0<\theta < \delta$
\[
\norm{A_{2}(t)\izt G(t,\sigma)f(\cdot,\sigma)d\sigma}_{H^{(1+\al)\theta}(\ve,\omega)} \leq \izt\norm{A_{2}(t)G(t,\sigma)f(\cdot,\sigma)}_{H^{(1+\al)\theta}(\ve,\omega)}d\sigma
\]
\[
\leq \izt \frac{c}{(t-\sigma)^{1+\theta-\delta}}\norm{f(\cdot,\sigma)}_{{}_{0}H^{(1+\al)\delta}(0,1)}d\sigma
\leq \frac{cT^{\delta-\theta}}{\delta-\theta}\norm{f}_{L^{\infty}(0,T;{}_{0}H^{\al}(0,1))}.
\]
By the regularity of the initial condition we obtain
\[
\norm{A_{2}(t)G(t,0)v_{0}}_{H^{(1+\al)\theta}(0,1)} \leq  \frac{c}{t^{\theta}}\norm{v_{0}}_{\dd}.
\]
Thus, in view of formula (\ref{wzorv}), we get that for every $0<\gamma <\al$
\[
A_{2}(t)v \in L^{\infty}_{loc}(0,T;H^{\gamma}_{loc}(0,1)), \hd A_{2}(t)v \in L^{\frac{\al+1}{\al}}(0,T;H^{\gamma}_{loc}(0,1)),
\]
which implies
\[
\p^{\al}v_{x} \in L^{\infty}_{loc}(0,T;H^{\gamma}_{loc}(0,1)), \hd  \p^{\al}v_{x} \in L^{\frac{\al+1}{\al}}(0,T;H^{\gamma}_{loc}(0,1)).
\]
Applying Lemma \ref{local} we obtain that
\[
v_{x} \in L^{\infty}_{loc}(0,T;H^{\gamma+\al}_{loc}(0,1)) \m{ and } v_{x} \in L^{\frac{\al+1}{\al}}(0,T;H^{\gamma+\al}_{loc}(0,1)),
\]
which finishes the proof in the case $\al \in (\frac{1}{2},1)$.\\
In the case $\al \leq \frac{1}{2}$ we have $f \in L^{\infty}(0,T;[L^{2},\dd]_{\delta})$. Thus, by (\ref{r2}), we obtain that for any $\theta < \delta$
\[
\norm{A_{2}(t)\izt G(t,\sigma)f(\cdot,\sigma)d\sigma}_{[L^{2}(0,1),\dd]_{\theta}} \leq \izt\norm{A_{2}(t)G(t,\sigma)f(\cdot,\sigma)}_{[L^{2}(0,1),\dd]_{\theta}}d\sigma
\]
\[
\leq \izt \frac{c}{(t-\sigma)^{1+\theta-\delta}}\norm{f(\cdot,\sigma)}_{[L^{2}(0,1),\dd]_{\delta}}d\sigma
\leq \frac{cT^{\delta-\theta}}{\delta-\theta}\norm{f}_{L^{\infty}(0,T;_{0}H^{\al}(0,1))}.
\]
This together with the estimate
\[
\norm{A_{2}(t)G(t,0)v_{0}}_{[L^{2}(0,1), \da]_{\theta}} \leq  \frac{c}{t^{\theta}}\norm{v_{0}}_{\dd}
\]
and formula (\ref{vcf}), implies that for every $0<\theta < \delta$ we have
\[
A_{2}(t)v \in L^{\infty}_{loc}(0,T;[L^{2}(0,1),\dd]_{\theta}) =L^{\infty}_{loc}(0,T;H^{(1+\al)\theta}(0,1)).
\]
Hence,
\[
\p^{\al}v_{x} \in L^{\infty}_{loc}(0,T;H^{\gamma_{0}}(0,1))  \m{ for every } \gamma_{0} \in (0,\al)
\]
and thus
\eqq{
\norm{v_{x}(\cdot,t)}_{{}_{0}H^{\gamma_{1}}(0,1)} \leq c t^{-\frac{\al}{\al+1}}\norm{v_{0}}_{\dd}
 \m{ for every } \gamma_{1} < 2\al.
}{vx02}
 Let us denote $\delta_{1} = \frac{\gamma_{1}}{1+\al}$. We will discuss firstly the case $\al \in (\frac{1}{4},\frac{1}{2}]$. We note that by~(\ref{vx02})
 \[
 \norm{f(\cdot,t)}_{H^{(1+\al)\delta_{1}}(0,1)} \leq c t^{-\frac{\al}{\al+1}}\norm{v_{0}}_{\dd}.
 \]
Applying Lemma \ref{lemusi2}, we obtain for every $\theta < \delta_{1}$
\[
\norm{A_{2}(t)v(\cdot,t)}_{H^{(1+\al)\theta}(\ve,\omega)} \leq \frac{c}{t^{\theta}} \norm{v_{0}}_{\dd} + \izt\norm{A_{2}(t)G(t,\sigma)f(\cdot,\sigma)}_{H^{(1+\al)\theta}(\ve,\omega)}d\sigma
\]
\[
\leq \frac{c}{t^{\theta}} \norm{v_{0}}_{\dd} + c \izt \sigma^{-\frac{\al}{\al+1}}(t-\sigma)^{-1-\theta+\delta_{1}}d\sigma\norm{v_{0}}_{\dd}
\]
and we arrive at
\[
A_{2}(t)v \in L^{\infty}_{loc}(0,T;H^{(1+\al)\theta}_{loc}(0,1)),
\hd A_{2}(t)v \in L^{\frac{\al+1}{2\al}}(0,T;H^{(1+\al)\theta}_{loc}(0,1))
 \m{ for every } \theta < \delta_{1}.
\]
Thus,
\[
\p^{\al}v_{x} \in L^{\infty}_{loc}(0,T;H^{\gamma_{1}}_{loc}(0,1)), \hd
\p^{\al}v_{x} \in L^{\frac{\al+1}{2\al}}(0,T;H^{\gamma_{1}}_{loc}(0,1))
 \m{ for every }\gamma_{1} < 2\al.
\]
Applying Lemma \ref{local} we get that
\[
v_{x} \in L^{\infty}_{loc}(0,T;H^{\gamma_{2}}_{loc}(0,1)),
\hd v_{x} \in L^{\frac{\al+1}{2\al}}(0,T;H^{\gamma_{2}}_{loc}(0,1)) \m{ for every } \gamma_{2} < 3\al.
\]
This way we proved the lemma for $\al \in (\frac{1}{3},1)$. If $\al \in (\frac{1}{4},\frac{1}{3}]$, since $f \in L^{\frac{\al+1}{2\al}}(0,T;H^{\gamma_{2}}_{loc}(0,1)\cap H^{\bar{\gamma}}(0,1))$ for every $0<\bar{\gamma}< \frac{1}{2}$, we apply Lemma~\ref{lemusi2} with $\beta = \gamma_{2}$ together with Lemma~\ref{local} and we obtain that
\[
\p^{\al}v_{x} \in L^{\infty}_{loc}(0,T;H^{\gamma_{2}}_{loc}(0,1)) \m{ and } v_{x} \in L^{\infty}_{loc}(0,T;H^{\gamma_{3}}_{loc}(0,1))  \m{ for every } \gamma_{3} < 4\al.
\]
In general case, we proceed as follows. For $\al \in (\frac{1}{(k+1)},\frac{1}{k}]$, $k\geq 2$ we apply to (\ref{vcf}) the estimate (\ref{r2}) with $\delta_{n} = \frac{\gamma_{n}}{\al+1}$, $\gamma_{n} < (n+1)\al$ for $n=0,\dots,\lceil\frac{k-1}{2}\rceil-1$. This way we obtain that
\eqq{
v_{x} \in L^{\infty}_{loc}(0,T;{}_{0}H^{\gamma_{\lceil(k-1)/2\rceil}}(0,1))
}{nowez}
and
\[
\norm{f(\cdot,t)}_{H^{\gamma_{\lceil(k-1)/2\rceil}}(0,1)} \leq c t^{-\frac{\gamma_{\lceil(k-1)/2\rceil}}{\al+1}} \norm{v_{0}}_{\dd}.
\]
Then we apply to (\ref{vcf}) Lemma \ref{lemusi2} together with Lemma \ref{local} $\lceil\frac{k}{2}\rceil$ times with $\beta = \gamma_{n}$ for $n=\lceil\frac{k-1}{2}\rceil, \dots, k-1$ to obtain
\[
\p^{\al}v_{x} \in  L^{\infty}_{loc}(0,T;H^{\gamma_{k-1}}_{loc}(0,1)), \hd v_{x} \in L^{\infty}_{loc}(0,T;H^{\gamma_{k}}_{loc}(0,1)), \hd v_{x} \in L^{\frac{\al+1}{k\al}}(0,T;H^{\gamma_{k}}_{loc}(0,1)),
\]
which finishes the proof.
\end{proof}

\no In Theorem \ref{first} we have obtained the solution to (\ref{niecv}) belonging to $C([0,T];\dd)$. By Lemma \ref{third} we may deduce local continuity of the solution with values in more regular spaces.
We establish this result in the following corollary.

\begin{coro}\label{regost}
  Let us assume that $v_{0} \in \dd$. Let $v$ be a solution to (\ref{niecv}) given by Theorem \ref{first}. Let $\al \in (0,1)$, we choose $k \in \mathbb{N} \setminus \{0\}$ such that  $\al \in (\frac{1}{k+1},\frac{1}{k}]$. Then, for every $\al<\gamma_{k}<(k+1)\al$ there holds
  \eqq{
  v \in C((0,T];H^{\gamma_{k}+1}_{loc}(0,1)) \m{ and } \p^{\al}v_{x} \in C((0,T];H^{\gamma_{k}-\al}_{loc}(0,1)).
  }{ostc}
  Furthermore,
  \eqq{
  v_{x} \in C((0,T];C[0,1]) \m{ for }\al \in (0,\frac{1}{2}] \m{ and } v_{x} \in C([0,T];C[0,1]) \m{ for }\al \in (\frac{1}{2},1).
  }{vxcc}
\end{coro}
\begin{proof}
Theorem \ref{first} states that $v \in C([0,T];\dd)$. Since for arbitrary $0<\ve<\omega<1$ and for every $\al<\overline{\gamma_{k}} < \gamma_{k}<(k+1)\al$ there holds
\[
H^{\overline{\gamma_{k}}+1}(\ve,\omega) = [H^{1+\al}(\ve,\omega),H^{\gamma_{k}+1}(\ve,\omega)]_{\frac{\overline{\gamma_{k}}-\al}{\gamma_{k}-\al}},
\]
we may estimate by the interpolation theorem (\cite[Corollary 1.2.7]{Lunardi})
\[
\norm{v(\cdot,t)-v(\cdot,\tau)}_{H^{\overline{\gamma_{k}}+1}(\ve,\omega)}
 \leq c\norm{v(\cdot,t)-v(\cdot,\tau)}_{\dd}^{1-\frac{\overline{\gamma_{k}}-\al}{\gamma_{k}-\al}}
\norm{v(\cdot,t)-v(\cdot,\tau)}_{H^{\gamma_{k}+1}(\ve,\omega)}^{\frac{\overline{\gamma_{k}}-\al}{\gamma_{k}-\al}},
\]
where $c =  c(\gamma_{k},\overline{\gamma_{k}},\ve,\omega)$.
By Lemma \ref{third}, the second norm on the right hand side above is bounded on every compact interval contained in $(0,T]$, while the first tends to zero as $\tau\rightarrow t$ for $t,\tau \in [0,T]$.

\no In order to obtain the claim for $\p^{\al}v_{x}$ we recall that by Theorem \ref{first} we have $\p^{\al}v_{x} \in C([0,T];L^{2}(0,1))$
Applying again the interpolation theorem we obtain for every $0<\ve<\omega<1$, $0<\tau < t \leq T$ and every $\al<\overline{\gamma_{k}} < \gamma_{k} < (k+1)\al$
\[
\norm{\p^{\al}v_{x}(\cdot,t) - \p^{\al}v_{x}(\cdot,\tau)}_{H^{\overline{\gamma_{k}}-\al}(\ve,\omega)}
\]
\[
\leq c(\gamma_{k},\overline{\gamma_{k}},\al,\ve,\omega)\norm{\p^{\al}v_{x}(\cdot,t) - \p^{\al}v_{x}(\cdot,\tau)}_{L^{2}(0,1)}^{1-\frac{\overline{\gamma_{k}}-\al}{\gamma_{k}-\al}}\norm{\p^{\al}v_{x}(\cdot,t) - \p^{\al}v_{x}(\cdot,\tau)}_{H^{\gamma_{k}-\al}(\ve,\omega)}^{\frac{\overline{\gamma_{k}}-\al}{\gamma_{k}-\al}}.
\]
The first norm tends to zero as $\tau \rightarrow t$, while the second one is bounded on every compact interval contained in $(0,T]$ due to Lemma \ref{third}. This way we proved~(\ref{ostc}). The continuity of $v_{x}$ in the case $\al \in (\frac{1}{2},1)$ follows by the Sobolev embedding from $v \in C([0,T],\dd)$. In the case $\al \in (0,\frac{1}{2}]$ we recall that $v_{x} \in C([0,T];L^{2}(0,1))$ and by~(\ref{nowez}) $v_{x} \in L^{\infty}_{loc}(0,T;H^{\gamma}(0,1))$, for a $\gamma > \frac{1}{2}$. Hence, applying again the interpolation argument together with Sobolev embedding, we arrive at (\ref{vxcc}).
\end{proof}
\begin{coro}\label{w2b}
Let us assume that $v_{0} \in \dd$. Let $v$ be a solution to (\ref{niecv}) given by Theorem \ref{first}. Then, for every $\al \in (0,1)$ there exists $\beta \in (\al,1)$ such that for every $0<\ve<\omega<1$ there holds $v~\in C((0,T];W^{2,\frac{1}{1-\beta}}(\ve,\omega))$.
\end{coro}
\begin{proof}
In the case $\al \in (0,\frac{1}{2})$ it is enough to notice that in view of Corollary \ref{regost} we have $v \in C((0,T];H^{2}(\ve,\omega))$ for every $0<\ve<\omega<1$. In the case $\al \in [\frac{1}{2},1)$ the claim follows from Corollary~\ref{regost} by the Sobolev embedding.
\end{proof}
\subsection{The existence and regularity of solutions to (\ref{niech})}
\no At last, we are ready to formulate and prove the result concerning the unique existence and regularity of solution to (\ref{niech}).
\begin{theorem}\label{existance}
Let $b,T > 0$ and $\al \in (0,1)$. Let us assume that $s$ satisfies (\ref{zals}). We further assume, that  $\uz \in H^{1+\al}(0,b)$, $\uz' \in {}_{0}H^{\al}(0,b)$ and $\uz(b)=0$. Then, there exists a unique solution $u$ to~(\ref{niech}) such that $u \in C(\overline{Q_{s,T}})$, $u_{t},\poch \da u \in C(Q_{s,T})$. Moreover, in the case $\al \in (\frac{1}{2},1)$ $u_{x} \in C(\overline{Q_{s,T}})$, while in the case $\al \in (0,\frac{1}{2}]$ $u_{x} \in C(\overline{Q_{s,T}} \setminus (\{t=0\} \times [0,b]))$. Finally, there exists $\beta \in (\al,1)$ such that for every $t \in (0,T]$ and every $0<\ve<\omega<s(t)$ we have $u(\cdot,t) \in W^{2,\frac{1}{1-\beta}}(\ve,\omega)$.
\end{theorem}
\begin{proof}
Firstly we will establish the results concerning the existence and regularity of solution to (\ref{niecv}) and then, we will rewrite the results in terms of properties of solution to (\ref{niech}).
We note that, under assumptions concerning regularity and traces of $\uz$ we obtain that $v_{0}$  defined in (\ref{defv0}) belongs to $\dd$.
Due to Theorem~\ref{first} and further regularity results given by Lemma \ref{second}, Lemma \ref{third} and Corollary~\ref{regost}, we obtain the existence of $v$ a unique solution to (\ref{niecv}). The solution satisfies $v \in C([0,T];\dd)$, which by the Sobolev embedding implies that $v \in C([0,T]\times [0,1])$. Furthermore, from Corollary~\ref{w2b} we know that there exists $\beta \in (\al,1)$ such that $v \in C((0,T];W^{2,\frac{1}{1-\beta}}(\ve,\omega))$ for every $0<\ve<\omega<1$.

\no We define the function $u$ on $Q_{s,T}$ by the formula $u(x,t) = v(\frac{x}{s(t)},t)$. Then, from
 (\ref{defv}) we infer that $u$ is a unique solution to (\ref{niech}). Since $v \in C([0,T]\times [0,1])$, we obtain that $u \in C(\overline{Q_{s,T}})$ and $v \in C((0,T];W^{2,\frac{1}{1-\beta}}(\ve,\omega))$  implies $u(\cdot,t) \in W^{2,\frac{1}{1-\beta}}(\ve,\omega)$ for every $t \in (0,T]$ and every $0<\ve<\omega<s(t)$.
 We note that $v_{p}(p,t) = s(t)u_{x}(x,t)$. Hence, from (\ref{vxcc}) we obtain that $u_{x} \in C(\overline{Q_{s,T}})$ in the case $\al \in (\frac{1}{2},1)$ and for $\al \in (0,\frac{1}{2}]$ we get $u_{x} \in C(\overline{Q_{s,T}} \setminus (\{t=0\} \times [0,b]))$.
\no On the other hand
\[
u_{t}(x,t) = \poch \da u(x,t) = \frac{1}{s^{1+\al}(t)} \frac{\p}{\p p} \da v(p,t) \m{ where } p=\frac{x}{s(t)}.
\]
From Corollary \ref{regost} and the Sobolev embedding we may deduce that $\p^{\al}v_{p} =\frac{\p}{\p p} \da v  \in C((0,T] \times (0,1))$, which implies $\poch \da u,\hd u_{t} \in C(Q_{s,T})$.
\end{proof}

\section{A solution to Stefan problem}
Before we prove the existence and uniqueness of the solution to Stefan problem, we need to derive the weak extremum principle for the system (\ref{niech}).
\subsection{Extremum principles}
 We will begin with the auxiliary lemmas. Firstly, we will present the extended version of \cite[Lemma 1]{decay} (see also \cite[Theorem 1]{Luchko}).
\begin{lem}\label{maxda}
Let us assume that $f:[0, L]\rightarrow \mathbb{R}$ belongs to $W^{1,\frac{1}{1-\beta}}(0,L)$ for some $\beta \in (0, 1]$. Then,
\begin{enumerate}
\item if $f$ attains its maximum over the interval $[0, L]$ at the point $x_{0} \in (0, L]$, then for every $\al \in (0,\beta)$ there holds the inequality $(\da f)(x_{0}) \geq 0$. Furthermore, if $f$ is not constant on $[0,x_{0}]$, then $(\da f)(x_{0}) > 0$.
\item If $f$ attains its minimum over the interval $[0, L]$ at the point $x_{0} \in (0, L]$, then for every $\al \in (0,\beta)$ there holds the inequality $(\da f)(x_{0}) \leq 0$. Furthermore, if $f$ is not constant on $[0,x_{0}]$, then $(\da f)(x_{0}) < 0$.
\end{enumerate}
\end{lem}
\begin{proof}
Let us assume that $f$ attains its maximum at the point $x_{0} \in (0, L]$. We define the function $g(x):=f(x_{0}) - f(x)$ for $x \in [0,L]$.  We note that $g(x) \geq 0, g(x_{0})=0$ and $(D^{\al}g)(x) = -(D^{\al}f)(x)$ for $x \in [0,L]$. For $x \in [0,x_{0}]$ we may estimate $g$ as follows
\eqq{
g(x) \leq \int^{x_{0}}_{x}\abs{g'(p)}dp \leq \norm{g'}_{L^{\frac{1}{1-\beta}}(0,L)}\abs{x-x_{0}}^{\beta}.
}{nbeta}
Thus, for fixed $\al \in (0,\beta)$, applying integration by parts formula, we get
\[
(D^{\al}g)(x_{0}) = \lim_{h\rightarrow 0^{+}}\frac{1}{\Gamma(1-\al)} \int_{0}^{x_{0}-h}(x_{0}-p)^{-\al}g'(p)dp
\]
\[
 = \lim_{h\rightarrow 0^{+}}\frac{h^{-\al}g(x_{0}-h)}{\Gamma(1-\al)} - \frac{x_{0}^{-\al}g(0)}{\Gamma(1-\al)} - \frac{\al}{\Gamma(1-\al)} \lim_{h\rightarrow 0^{+}}\int_{0}^{x_{0}-h}(x_{0}-p)^{-\al-1}g(p)dp.
\]
From the estimate (\ref{nbeta}) we infer that the first limit equals zero. Applying the Lebesgue monotone convergence theorem we obtain that
\eqq{
(D^{\al}g)(x_{0}) = - \frac{x_{0}^{-\al}g(0)}{\Gamma(1-\al)} - \frac{\al}{\Gamma(1-\al)} \int_{0}^{x_{0}}(x_{0}-p)^{-\al-1}g(p)dp.
}{lmm}
Thus $(\da g)(x_{0})\leq 0$, which is equivalent with $(\da f)(x_{0})\geq 0$. Furthermore, from the formula (\ref{lmm}) we obtain that if $f$ is not a constant function on $[0,x_{0}]$ then $(\da f)(x_{0})~>0$. Substituting  $f$ by $-f$ we obtain the second part of the claim.
\end{proof}

In the next lemma we will show that $\poch \da f$ is non positive in the maximum point of $f$ in the interior of the interval. This result, under stronger regularity assumptions, was proved in \cite[Lemma 2.2]{visco}. Here we present the proof, where we do not demand $C^{2}$ regularity of $f$.
\begin{lem}\label{nonpositivity}
Let $f:[0,L]\rightarrow \mathbb{R}$ be such that $f' \in AC[0,L]$ and for every $\kappa > 0$ $f'' \in L^{\frac{1}{1-\beta}}(\kappa,L)$ for fixed $\beta \in (0,1)$. If $f$ attains its local maximum in $x_{0}\in (0,L)$ which is a global maximum on $[0,x_{0}]$, then  $(\poch \da f)(x_{0}) \leq 0$ for every $\al \in (0,\beta)$.
\end{lem}
\begin{proof}
We define $g(x) = f(x_{0})-f(x)$. Then $g$ is nonnegative on $[0,x_{0}]$, $g'(x_{0})=0$ and $\poch\da g = - \poch \da f$. Let us fix $\kappa \in (0,x_{0})$. We note that for $x\in(\kappa,x_{0})$ we may estimate
\eqq{
\abs{g'(x)}\leq \int_{x}^{x_{0}}\abs{ g''(p)}dp \leq \norm{g''}_{L^{\frac{1}{1-\beta}}(\kappa,L)}\abs{x-x_{0}}^{\beta}
}{gp}
and
\[
g(x) \leq \int_{x}^{x_{0}}\abs{ g'(p)}dp \leq \int_{x}^{x_{0}}\int_{p}^{x_{0}}\abs{g''(r)}drdp
\]
\eqq{
\leq \norm{g''}_{L^{\frac{1}{1-\beta}}(\kappa,L)}\int_{x}^{x_{0}} \abs{p-x_{0}}^{\beta}dp = \norm{g''}_{L^{\frac{1}{1-\beta}}(\kappa,L)}\frac{\abs{x-x_{0}}^{\beta+1}}{\beta+1}.
}{gbezp}
Making use of these estimates we may differentiate under the integral sign as follows
\[
(\poch \da g)(x_{0}) = \frac{1}{\Gamma(1-\al)}\left(\poch\int_{0}^{x}(x-p)^{-\al}g'(p)dp\right)(x_{0})
\]
\[
\frac{1}{\Gamma(1-\al)}\left(\poch\int_{0}^{\kappa}(x-p)^{-\al}g'(p)dp\right)(x_{0})+
\frac{1}{\Gamma(1-\al)}\left(\poch\int_{\kappa}^{x}(x-p)^{-\al}g'(p)dp\right)(x_{0})
\]
\[
=-\frac{\al}{\Gamma(1-\al)}\int_{0}^{\kappa}(x_{0}-p)^{-\al-1}g'(p)dp
-\frac{\al}{\Gamma(1-\al)}\int_{\kappa}^{x_{0}}(x_{0}-p)^{-\al-1}g'(p)dp
\]
\[
+\frac{1}{\Gamma(1-\al)}\lim_{p\rightarrow x_{0}}(x_{0}-p)^{-\al}g'(p)
\]
and the last limit is equal to zero by the estimate (\ref{gp}).
Applying integration by parts we get
\[
-\frac{\al}{\Gamma(1-\al)}\int_{\kappa}^{x_{0}}(x_{0}-p)^{-\al-1}g'(p)dp =
 -\frac{\al}{\Gamma(1-\al)}\lim_{p\rightarrow x_{0}}(x_{0}-p)^{-\al-1}g(p)
 \]
 \[
 + \frac{\al}{\Gamma(1-\al)}(x_{0}-\kappa)^{-\al-1}g(\kappa) + \frac{\al(\al+1)}{\Gamma(1-\al)}\int_{\kappa}^{x_{0}}(x_{0}-p)^{-\al-2}g(p)dp.
\]
By (\ref{gbezp}) the limit equals zero.  For a fixed $\ve > 0$ we may choose $\kappa>0$ such that
\[
\abs{-\frac{\al}{\Gamma(1-\al)}\int_{0}^{\kappa}(x_{0}-p)^{-\al-1}g'(p)dp} \leq \ve.
\]
Thus we obtain
\[
(\poch \da g)(x_{0}) \geq -\ve  + \frac{\al}{\Gamma(1-\al)}(x_{0}-\kappa)^{-\al-1}g(\kappa) + \frac{\al(\al+1)}{\Gamma(1-\al)}\int_{\kappa}^{x_{0}}(x_{0}-p)^{-\al-2}g(p)dp.
\]
Since $\ve > 0$ was arbitrary we arrive at the estimate
\[
(\poch \da g)(x_{0}) \geq 0, \m{ which implies } (\poch \da f)(x_{0}) \leq 0.
\]
\end{proof}
Having proven Lemma \ref{nonpositivity}, it is not difficult to deduce the weak extremum principle for parabolic-type problems involving $\poch \da$.
\begin{lem}[Weak extremum principle] \label{weakmax}
We assume that $u$ satisfies
\[
u_{t} - \poch \da u = f \m{ in } Q_{s,T}
\]
and has the following regularity $u \in C(\overline{Q_{s,T}})$, $u_{t}\in C(Q_{s,T})$ and for every $t\in (0,T)$, for every $0<\ve<\omega<s(t)$ we have $u(\cdot,t) \in W^{2,\frac{1}{1-\beta}}(\ve,\omega)$ for some $\beta \in (\al,1]$. Let us denote the parabolic boundary of $Q_{s,T}$ by $\p \Gamma_{s,T} = \p \overline{Q_{s,T}} \setminus (\{T\} \times (0,s(T)))$. Then,
\begin{enumerate}
\item
if $f \leq 0$, then $u$ attains its maximum on $\p \Gamma_{s,T}$.
\item If $f \geq 0$, then $u$ attains its minimum on $\p \Gamma_{s,T}$.
\end{enumerate}
\end{lem}
\begin{proof}
The proof follows the standard argument for the linear parabolic equations. Firstly, we will prove the first part of the lemma. Let us assume that at some point $(x_{0},t_{0}) \in \overline{Q_{s,T}} \setminus \p \Gamma_{s,T} $ we have $u(x_{0},t_{0}) > \max_{\p \Gamma_{s,T}}=:M.$ We fix $\ve > 0$ and we denote $v(x,t) = (u(x,t)-M)e^{-\ve t}$ Then $v$ attains its positive maximum in some point $(x_{1},t_{1}) \in Q_{s,T}$. We may calculate
\[
v_{t} = u_{t}e^{-\ve t} - \ve v, \hd \hd \hd \poch \da v = e^{-\ve t}\poch \da u.
\]
Thus
\[
v_{t} - \poch \da v = -\ve v + fe^{-\ve t}.
\]
In particular
\[
v_{t}(x_{1},t_{1}) - \poch \da v(x_{1},t_{1}) = -\ve v(x_{1},t_{1}) + f(x_{1},t_{1})e^{-\ve t_{1}} < 0.
\]
Since $(x_{1},t_{1})$ is a maximum point we have $v_{t}(x_{1},t_{1})\geq 0$ and by Lemma \ref{nonpositivity} we infer that $\poch \da v(x_{1},t_{1}) \leq 0$. Hence, $v_{t}(x_{1},t_{1}) - \poch \da v(x_{1},t_{1}) \geq 0$, which  leads to a contradiction. Setting $u:=-u$ we obtain the second part of the claim.
\end{proof}

\subsection{Estimates}
In the next two lemmas, we derive the bounds for the Caputo derivative of the solution to (\ref{niech}) and for the solution itself. This is a significant step in the proof of the existence of  solution to (\ref{Stefan}).
\begin{lem}\label{nonposda}
Let us assume that the assumptions of Theorem \ref{existance} are satisfied and additionally $\uz \geq 0$. Let $u$ be a solution to (\ref{niech}) given by Theorem \ref{existance}, then $(\da u)(s(t),t)~\leq~0$. Furthermore, if $\uz \not\equiv0$, then for every $t \in (0,T]$ we have $(\da u)(s(t),t)<0$.
\end{lem}
\begin{proof}
By Theorem \ref{existance} function $u$ satisfies the assumptions of Lemma \ref{weakmax}. Hence, it attains its minimum at the parabolic boundary. In order to show that the minimum is attained on the curve $(s(t),t)$ we introduce $u_{\ve} = u - \ve x$. Then $u_{\ve}$ satisfies
 \[
 \left\{ \begin{array}{ll}
u_{\ve t} - \poch \da u_{\ve} = \frac{\ve x^{-\al}}{\Gamma(1-\al)} & \textrm{ in }  Q_{s,T}, \\
u_{\ve x}(0,t) = -\ve, u_{\ve}(s(t),t) = -\ve s(t)  & \textrm{ for  } t \in (0,T), \\
u_{\ve}(x,0) =\uz(x) -\ve x & \textrm{ for } 0<x<b. \\
\end{array} \right. \]

\no From Lemma \ref{weakmax} we deduce that $u_{\ve}$ also attains its minimum on the parabolic boundary and due to $u_{\ve,x}(0,t)<0$ we obtain that
\[
u_{\ve}(x,t) \geq  \min\{\uz(x)-\ve x, -\ve s(t)\} \geq -\ve s(t),
\]
where we used the assumption $\uz \geq 0$.
Hence, $u(x,t) = u_{\ve}(x,t)+\ve x \geq -\ve s(t)$. Passing to the limit with $\ve$ we obtain that $u \geq 0$.
Hence, $u$ attains its minimum, which is equal to zero, on the curve $(s(t),t)$. Applying the minimum principle in spatial dimension (Lemma \ref{maxda}), we obtain that $(\da u)(s(t),t) \leq 0$ for every $t\in [0,T]$.\\

It remains to show that if $\uz \not\equiv 0$, then $(\da u)(s(t),t) <0$ for every $t \in (0,T]$. In the proof we will employ the ideas introduced in \cite[Appendix 2, Lemma 2.1]{Andreucci}. We will proceed by contradiction. Let us assume that for fixed $t_{0}>0$ we have $(\da u)(s(t_{0}),t_{0}) =0$. Then, by Lemma \ref{maxda} we infer that $u(x,\tz) = 0$ for every $x \in [0,s(t_{0})]$. We will prove that this leads to $u\equiv0$ on $Q_{s,t_{0}}$, which together with continuity of $u$ contradicts $\uz \not\equiv 0$. Let us assume that $u\not\equiv0$ on $Q_{s,t_{0}}$. Then, by continuity of $u$, we may choose  $0<\tj < \tz$, $x_{1} \in (0,s(\tj))$ and small $\delta > 0$, such that $u(x,\tj)>0$ for every $x$ belonging to $[x_{1},x_{1}+2\delta]$.
 \\  We introduce nonnegative auxiliary function $\eta: [0, \xj + 2\delta] \times [\tj,\tz]  \rightarrow \mathbb{R}$ as follows
\[
\eta(x,t) = \left\{ \begin{array}{ll}
0 & \textrm{ on }  [0, \xj] \times [\tj,\tz], \\
 \ve e^{-a(t-\tj)}[\delta^{2}-(x-\xj-\delta)^{2}]^{2}  & \textrm{ on  } (\xj, \xj + 2\delta] \times [\tj,\tz], \\
\end{array} \right.\]
where the constant $a > 0$ will be chosen later and $\ve>0$ is chosen in such a way that
\[
 \ve[\delta^{2}-(x-\xj-\delta)^{2}]^{2} \leq u(x,\tj) \m{ for every  } x \in (\xj, \xj + 2\delta).
\]
Such a choice of $\ve > 0$ is possible, if $\delta> 0$ is small, due to the continuity of $u$ and the fact that $u(\xj,\tj)>0$. Since $\eta(\xj,t) = \eta_{x}(\xj,t) = 0$, it is easy to notice that  $\eta$ satisfies regularity assumptions of Lemma \ref{weakmax} on $[0, \xj + 2\delta] \times [\tj,\tz] $. Furthermore, we have
\eqq{\eta(0,t) = \eta(\xj+2\delta,t)=0 \m{ for every } t \in [\tj,\tz]. }{e2}
By the assumption concerning $\ve$ there holds
\eqq{\eta(x,\tj)  \leq u(x,\tj) \m{ for every  } x \in [0, \xj + 2\delta].}{e3}
Our aim is to apply the weak minimum principle, obtained in Lemma \ref{weakmax}, to the function $w:=u-\eta$. To this end, we will show that for suitably chosen $a > 0$ we have
\eqq{-\eta_{t} + \poch \da \eta \geq 0 \m{ in }  (0, \xj + 2\delta) \times (\tj,\tz].} {e4}
At first we note that, by the definition of $\eta$ we have
\[
-\eta_{t} + \poch \da \eta \equiv 0 \m{ on } (0, \xj] \times (\tj,\tz] .
\]
We note that for $x > \xj$ we may write
 \[
 (\poch \da \eta)(x,t) =  \frac{1}{\Gamma(1-\al)}\poch\int_{\xj}^{x}(x-p)^{-\al}\eta_{x}(p,t)dp=: (\poch \da_{\xj}\eta)(x,t).
 \]
In order to calculate $\poch \da_{\xj}\eta$ we note that $\eta_{x}(\xj,t) = 0$, thus $\poch \da_{\xj}\eta =  \da_{\xj}\eta_{x}$.
Let us perform the calculations. We have
\[
\eta_{x}(x,t) = -4\ve e^{-a(t-\tj)}[\delta^{2}-(x-\xj-\delta)^{2}](x-\xj-\delta)
\]
and
\[
\eta_{xx}(x,t) =  -4\ve e^{-a(t-\tj)}(\delta^{2}-3(x-\xj-\delta)^{2}).
\]
Thus, we may write
\[
\poch \da_{\xj}\eta = \frac{4\ve e^{-a(t-\tj)}}{\Gamma(1-\al)}\left[3\int_{\xj}^{x}(x-p)^{-\al}(p-\xj-\delta)^{2}dp - \delta^{2}\int_{\xj}^{x}(x-p)^{-\al}dp\right].
\]
Calculating the last integral we obtain, that for $(x,t) \in (\xj,\xj+2\delta) \times (\tj,\tz)$ there holds
\[
-\eta_{t} + \poch \da \eta =
\]
\eqq{
\ve e^{-a(t-\tj)}\left(a[\delta^{2}-(x-\xj-\delta)^{2}]^{2} +
\frac{4}{\Gamma(1-\al)}\left[3\int_{\xj}^{x}(x-p)^{-\al}(p-\xj-\delta)^{2}dp - \frac{\delta^{2}(x-\xj)^{1-\al}}{1-\al}\right]
\right).
}{szal}
We will show that the last expression is nonnegative for every $(x,t) \in (\xj,\xj+2\delta) \times (\tj,\tz)$ for suitably chosen $a>0$.
At first, we note that
\eqq{
\kappa_{\al}:=\frac{1}{2-\al}\left[3 - \sqrt{3}\sqrt{\frac{1+\al}{3-\al}}\right] > 1 \m{ for every } \al \in (0,1).}
{kal}
Let us introduce
\eqq{
\omega_{\al,\delta}:=\frac{2\delta (\kappa_{\al}-1)}{\kappa_{\al}}.
}{omal}
We will consider three cases.\\
1. Let $x\in[\xj+\frac{1}{3}\delta,\xj + 2\delta -\omega_{\al,\delta}]$. Then,
\[
[\delta^{2}-(x-\xj-\delta)^{2}]^{2}\geq [\delta^{2}-(\delta - \omega_{\al,\delta})^{2}]^{2} \m{ and } (x-\xj)^{1-\al} \leq \left(2\delta - \omega_{\al,\delta}\right)^{1-\al}.
\]
Thus, for $a \geq \frac{4\delta^{2}}{\Gamma(2-\al)}\frac{(2\delta -  \omega_{\al,\delta})^{1-\al}}{[\delta^{2}-(\delta - \omega_{\al,\delta})^{2}]^{2}}$ we have
\[
a[\delta^{2}-(x-\xj-\delta)^{2}]^{2} \geq \frac{4\delta^{2}(x-\xj)^{1-\al}}{\Gamma(2-\al)}
\]
and the expression (\ref{szal}) is nonnegative.\\
2. If $x \in (\xj, \xj + \frac{1}{3}\delta]$, we may notice that
\[
3\int_{\xj}^{x}(x-p)^{-\al}(p-\xj-\delta)^{2}dp \geq 3\frac{4\delta^{2}}{9}\int_{\xj}^{x}(x-p)^{-\al}dp = \frac{4\delta^{2}}{3}\frac{(x-\xj)^{1-\al}}{1-\al},
\]
which ensures that (\ref{szal}) is nonnegative.\\
3. It remains to deal with the case $x \in [\xj + 2\delta - \omega_{\al,\delta}, \xj + 2 \delta)$. We apply the substitution $p = \xj + r(x-\xj)$ to obtain that
\[
3\int_{\xj}^{x}(x-p)^{-\al}(p-\xj-\delta)^{2}dp = 3 \izj (1-r)^{-\al}(r(x-\xj) - \delta)^{2}dr (x-\xj)^{1-\al}.
\]
Thus, it is enough to prove that for each $x \in [\xj + 2\delta - \omega_{\al,\delta}, \xj + 2 \delta]$
\[
3 \izj (1-r)^{-\al}(r(x-\xj) - \delta)^{2}dr \geq \frac{\delta^{2}}{1-\al},
\]
which is equivalent with
\[
3 \izj (1-r)^{-\al}r^{2}dr(x-\xj)^{2} - 6\delta\izj (1-r)^{-\al}rdr(x-\xj) + \frac{2\delta^{2}}{1-\al} \geq 0.
\]
Calculating the above integrals and dividing the inequality by $2$ we have
\[
\frac{\delta^{2}}{1-\al} -3\delta(x-\xj)\frac{\Gamma(1-\al)}{\Gamma(3-\al)} + 3(x-\xj)^{2}\frac{\Gamma(1-\al)}{\Gamma(4-\al)} \geq 0.
\]
Multiplying the inequality by $\frac{\Gamma(2-\al)}{\Gamma(1-\al)}$ we obtain another equivalent inequality
\[
\delta^{2} - \frac{3(x-\xj)}{2-\al}\delta + \frac{3(x-\xj)^{2}}{(2-\al)(3-\al)} \geq 0.
\]
By direct calculations we see that the roots of the function
\[
f(\delta):=\delta^{2} - \frac{3(x-\xj)}{2-\al}\delta + \frac{3(x-\xj)^{2}}{(2-\al)(3-\al)}
\]
are given by the formula
\[
\delta_{\mp} = \frac{(x-\xj)}{2(2-\al)}\left[3 \mp \sqrt{3}\sqrt{\frac{1+\al}{3-\al}}\right].
\]
Thus, it is enough to show that $\delta \leq \delta_{-}$ for every choice of $x \in [\xj + 2\delta - \omega_{\al,\delta}, \xj + 2 \delta]$.
Recalling the definitions (\ref{kal}) and (\ref{omal}), we have
\[
\delta_{-} = \kappa_{\al}\frac{(x-\xj)}{2} \geq  \kappa_{\al}\frac{2\delta - \omega_{\al,\delta}}{2} = \delta.
\]
This way we have shown that (\ref{szal}) is nonnegative for $x \in [\xj + 2\delta - \omega_{\al,\delta}, \xj + 2 \delta)$. Summing up the result, we obtained that (\ref{szal}) is nonnegative for every $x\in(\xj,\xj+2\delta)$ and, as a consequence, (\ref{e4}) holds.\\
Let us define $w = u-\eta$.
Then, applying (\ref{e2}), (\ref{e3}), (\ref{e4}) we obtain that
\[
 \left\{ \begin{array}{ll}
w_{t} - \poch D^{\al}_{\xj} w \geq 0 & \textrm{ in }  (0, \xj + 2\delta) \times (\tj,\tz], \\
w(0,t) = u(0,t) \geq 0, \ \ w(\xj+2\delta,t) = u(\xj+2\delta,t) \geq 0 & \textrm{ for  } t \in [\tj,\tz], \\
w(x,\tj) \geq 0 & \textrm{ for } x \in [0, \xj + 2\delta]. \\
\end{array} \right.
\]
Obviously Lemma \ref{weakmax} is true also if we consider a problem in a cylindrical domain, thus we may apply the minimum principle, to obtain that $w$ attains its minimum on the parabolic boundary of $[0, \xj + 2\delta] \times [\tj,\tz]$. Thus, $w \geq 0$ in  $[0, \xj + 2\delta] \times [\tj,\tz]$. In particular
\[
u(x,\tz) \geq \eta(x,\tz) = \ve e^{-a(\tz - \tj)}[\delta^{2}-(x-\xj-\delta)^{2}]^{2} > 0 \m{ for every } x \in (\xj,\xj+2\delta).
\]
This is a contradiction with $u(x,\tz) = 0$ on $[0,s(\tz)]$. Thus, we obtained that $u\equiv 0$ in $Q_{s,\tz}$, which is again a contradiction with $\uz \not\equiv 0$. This way we proved the lemma.
\end{proof}
We have just proven that $(\da u)(s(t),t)$ is non positive for every $t \in [0,T]$. In the next lemma, we will find the lower bound for  $(\da u)(s(t),t)$. We will also find a suitable estimate for a solution $u$.

\begin{lem}\label{daboundlem}
Let us assume that $\uz \geq 0$ satisfies the assumptions of Theorem \ref{existance}. Let us assume additionally that there exists $M>0$ such that
\eqq{
\uz(x)\leq \frac{M\Gamma(2-\al)}{b^{1-\al}}(b-x) \m{ for every } x \in [0,b].
}{Mu}
 Furthermore, let $s$ fulfill the assumption (\ref{zals}), where the constant $M$ comes from (\ref{Mu}). Let $u$ be a solution to (\ref{niech}) given by Theorem \ref{existance}. Then, there hold the following bounds,
\eqq{(\da u)(s(t),t) \geq -M \m{ for every } t \in (0,T)}{dabound}
and
\eqq{0 \leq u(x,t)\leq  M \Gamma(2-\al)s^{\al-1}(t)(s(t)-x) \m{ for } (x,t) \in Q_{s,T}.}{estu}
\end{lem}

\begin{rem}
We note that in the case $\al \in (\frac{1}{2},1)$ the assumption (\ref{Mu}) is trivial, since from $\uz \in H^{1+\al}(0,1)$ follows that $\uz$ is Lipschitz continuous.
\end{rem}

\begin{proof}

In the proof we follow the ideas introduced in \cite[Proposition 4.2]{Andreucci}, where the author consider the classical Stefan problem. We define an auxiliary function $v$ by the formula
\[
v(x,t) =M_{0} s^{\al-1}(t)(s(t)-x),
\]
where $M_{0} = M \Gamma(2-\al)$. Then we may calculate
\[
(\da v)(s(t),t) = -\frac{M_{0}s^{\al-1}(t)}{\Gamma(1-\al)}\int_{0}^{s(t)}(s(t)-p)^{-\al}dp = -\frac{M_{0}}{\Gamma(2-\al)} = -M.
\]
Moreover, making use of (\ref{Mu}) we obtain
\[
v(s(t),t) = 0, \hd \hd v_{x}(x,t) = -M_{0}s^{\al-1}(t) <  0=u_{x}(0,t), \hd \hd v(x,0) = \frac{M_{0}}{b^{1-\al}}(b-x) \geq u_{0}(x).
\]
We may calculate further
\[
v_{t}(x,t) = M_{0}\al s^{\al-1}(t)\dot{s}(t) + (1-\al)M_{0}s^{\al-2}(t)\dot{s}(t)x,
\]
\[
\poch \da v(x,t) = -\frac{M_{0}s^{\al-1}(t)}{\Gamma(1-\al)}x^{-\al}.
\]
Together we have
\[
v_{t}(x,t) - \poch \da v(x,t)
\]
\[
= M_{0}\al s^{\al-1}(t)\dot{s}(t) + (1-\al)M_{0}s^{\al-2}(t)\dot{s}(t)x+\frac{M_{0}s^{\al-1}(t)}{\Gamma(1-\al)}x^{-\al}=:-f(x,t)\geq 0.
\]
We define the function $w = u-v$. Then $w$ satisfies
\[
 \left\{ \begin{array}{ll}
w_{t} - \poch \da w = f & \textrm{ in }  Q_{s,T}, \\
w_{x}(0,t) > 0, \ \ w(s(t),t) = 0 & \textrm{ for  } t \in (0,T), \\
w(x,0) \leq 0 & \textrm{ for } 0<x<s(0). \\
\end{array} \right.
\]
We may apply the weak maximum principle from Lemma \ref{weakmax} to function $w$ to obtain that $\max_{\overline{Q_{s,T}}}w =\max_{\p\Gamma_{s,T}}w$. We note that $w(x,0) \leq 0$ and $w_{x}(0,t) > 0$ and $w(s(t),t) = 0$, thus $w \leq 0$ and we obtain (\ref{estu}). Moreover, $w$ must admit its maximum on the part of the boundary $(s(t),t)$, where it is equal to zero. Thus, by Lemma \ref{maxda}, we get $(\da w)(s(t),t) \geq 0$, thus $(\da u)(s(t),t) \geq (\da v)(s(t),t)~=~-M$.
\end{proof}

\subsection{A proof of the final result}
Finally, we are ready to prove the theorem concerning the existence and uniqueness of the regular solution~to~(\ref{Stefan}). At first, we will show the existence of the solution. The method of the proof relays on the construction of the free boundary $s(\cdot)$ by the Schauder fixed point theorem.
\no Subsequently, we show that the obtained solution is unique. It will be done by proving the monotone dependence of solutions upon data.
\begin{theorem}\label{final}
Let $b, T > 0$ and $\al \in (0,1)$. Let us assume that $\uz \in H^{1+\al}(0,b)$, $\uz' \in {}_{0}H^{\al}(0,b)$, $\uz(b)=0$ and $\uz \geq 0$, $\uz \not\equiv 0$. Further let us assume that there exists $M>0$ such that for every $x \in [0,b]$
\[
\uz(x)\leq \frac{M\Gamma(2-\al)}{b^{1-\al}}(b-x).
\]
\no Then, there exists $(u,s)$ a solution to (\ref{Stefan}), such that $s\in C^{0,1}([0,T])$, for almost all $t\in (0,T]$ there holds $0~<~\dot{s}(t)~\leq~M$, $u \in C(\overline{Q_{s,T}})$, $u_{t}, \poch \da u \in C(Q_{s,T})$. Moreover, in the case $\al \in (\frac{1}{2},1)$ $u_{x} \in C(\overline{Q_{s,T}})$, while in the case $\al \in (0,\frac{1}{2}]$ $u_{x} \in C(\overline{Q_{s,T}} \setminus (\{t=0\} \times [0,b]))$. Finally, there exists $\beta \in (\al,1)$, such that for every $t \in (0,T]$ and every $0<\ve<\omega<s(t)$ we have $u(\cdot,t)\in W^{2,\frac{1}{1-\beta}}(\ve,\omega)$.
\end{theorem}
\begin{proof}[Proof of Theorem \ref{final}]

We follow the idea introduced in the proof of \cite[Theorem 5.1]{Andreucci}.
We define the set
\[
\Sigma := \{s\in C^{0,1}[0,T], \hd 0< \dot{s}\leq M, \hd s(0)=b \}.
\]
Then $\Sigma$ is a compact and convex subset of a Banach space $C([0,T])$
with a maximum norm and for every $s\in \Sigma$ there exists a unique solution to (\ref{niech}), given by Theorem \ref{existance}. For $s\in \Sigma$ we define the operator
\[
(Ps)(t) = b - \izt (\da u)(s(\tau),\tau)d\tau,
\]
where $u$ is a solution to (\ref{niech}), corresponding to $s$, given by Theorem \ref{existance}. We would like to apply the Schauder fixed point theorem, thus we have to show that $P:\Sigma\rightarrow\Sigma$ and that it is continuous in maximum norm. Clearly we have $(Ps)(0) = b$ and from Lemma \ref{nonposda} and estimate (\ref{dabound}) we infer
\[
0 < \frac{d}{dt}(Ps)(t) = -(\da u)(s(t),t) \leq M.
\]
Hence, $P:\Sigma\rightarrow\Sigma$.

\no To prove that $P$ is continuous in maximum norm, we firstly note that integrating the first equation in (\ref{niech}) we obtain
\[
(\da u)(s(\tau),\tau) = \int_{0}^{s(\tau)}u_{t}(x,\tau)dx.
\]
Hence, we may rewrite the formula for $P$ as follows
\[
(Ps)(t) = b - \izt \int_{0}^{s(\tau)}u_{t}(x,\tau)dx d\tau  = b - \int_{0}^{b}\izt u_{t}(x,\tau) d\tau dx - \int_{b}^{s(t)} \int_{s^{-1}(x)}^{t}u_{t}(x,\tau)d\tau dx
\]
\eqq{
=b -\int_{0}^{b}u(x,t)dx + \int_{0}^{b}u(x,0)dx - \int_{b}^{s(t)}u(x,t)dx = b + \int_{0}^{b}\uz(x)dx - \int_{0}^{s(t)}u(x,t)dx.
}{pfo}

Now, we take arbitrary $s_{1}$, $s_{2} \in \Sigma$. Let us define $s_{min}(t)=\min\{s_{1}(t), s_{2}(t)\}$, $s_{max}(t)=\max\{s_{1}(t), s_{2}(t)\}$. We also define function $i=i(t) = 1$ if $s_{max}(t) = s_{1}(t)$ and $i = 2$ otherwise. Let $u_{1}$ and $u_{2}$ be two solutions to (\ref{niech}), given by Theorem \ref{existance}, corresponding to $s_{1}$ and $s_{2}$ respectively.
Let us define $v(x,t) = u_{1}(x,t)-u_{2}(x,t)$ and $v^{\ve}(x,t) = v(x,t)+\ve x$. Then $v^{\ve}$ satisfies
 \[
 \left\{ \begin{array}{ll}
v^{\ve}_{t} - \poch \da v^{\ve} = -\frac{\ve x^{-\al}}{\Gamma(1-\al)} & \textrm{ in } \{(x,t):0<x<s_{min}(t), 0<t<T\}=: Q_{s_{min},T}, \\
v^{\ve}_{x}(0,t) = \ve,  & \textrm{ for  } t \in (0,T), \\
v^{\ve}(x,0) = \ve x & \textrm{ in } 0<x<b. \\
\end{array} \right. \]

\no From Lemma \ref{weakmax} we obtain that $v^{\ve}$ attains its maximum on the parabolic boundary. We may estimate
\[
\abs{v^{\ve}(s_{min}(t),t)} \leq \abs{u_{1}(s_{min}(t),t)} + \abs{u_{2}(s_{min}(t),t)} + \ve s_{min}(T) = \abs{u_{i}(s_{min}(t),t)}+ \ve s_{min}(T)
\]
and since $v^{\ve}(x,0) = \ve x \leq \ve s_{min}(T)$ and $v_{x}^{\ve}(0,t) > 0$ we obtain that
\[
\max_{Q_{s_{min},T}}v^{\ve} \leq \abs{u_{i}(s_{min}(t),t)}+ \ve s_{min}(T).
\]
Applying the estimate (\ref{estu}) form Lemma \ref{daboundlem} we get
\[
\abs{u_{i}(s_{min}(t),t)} \leq M_{0}s_{max}^{\al-1}(t)(s_{max}(t)-s_{min}(t)) \leq M_{0}b^{\al-1}\max_{\tau\in [0,t]}\abs{s_{1}(\tau)-s_{2}(\tau)}.
\]
Hence,
\[
\max_{Q_{s_{min},T}} v = \max_{Q_{s_{min},T}}(v^{\ve}-\ve x) \leq M_{0}b^{\al-1}\max_{\tau\in [0,t]}\abs{s_{1}(\tau)-s_{2}(\tau)} + \ve s_{min}(T).
\]
Passing with $\ve$ to zero we obtain
\[
\max_{Q_{s_{min},T}}v \leq M_{0}b^{\al-1}\max_{\tau\in [0,t]}\abs{s_{1}(\tau)-s_{2}(\tau)}.
\]
To estimate $v$ from below we proceed similarly. We introduce $v_{\ve}(x,t) = v(x,t) - \ve x $. Then $v_{\ve}$ satisfies

 \[
 \left\{ \begin{array}{ll}
v_{\ve t} - \poch \da v_{\ve} = \frac{\ve x^{-\al}}{\Gamma(1-\al)} & \textrm{ in }  Q_{s_{min},T} \\
v_{\ve x}(0,t) = -\ve,  & \textrm{ for  } t \in (0,T) \\
v_{\ve}(x,0) = -\ve x & \textrm{ in } 0<x<b. \\
\end{array} \right. \]

\no Lemma \ref{weakmax} implies that $v_{\ve}$ attains its minimum on the parabolic boundary. We may estimate
\[
v_{\ve}(s_{min}(t),t) \geq  -\abs{u_{i}(s_{min}(t),t)}- \ve s_{min}(T)
\]
and since $v_{\ve}(x,0) = -\ve x \geq -\ve s_{min}(T)$ and $v_{\ve x}(0,t) < 0$ we obtain that
\[
\min_{Q_{s_{min},T}}v_{\ve} \geq -\abs{u_{i}(s_{min}(t),t)}- \ve s_{min}(T) \geq -M_{0}b^{\al-1}\max_{\tau\in [0,t]}\abs{s_{1}(\tau)-s_{2}(\tau)}- \ve s_{min}(T),
\]
thus
\[
\min_{Q_{s_{min},T}} v = \min_{Q_{s_{min},T}}(v_{\ve}+\ve x) \geq -M_{0}b^{\al-1}\max_{\tau\in [0,t]}\abs{s_{1}(\tau)-s_{2}(\tau)}-\ve s_{min}(T).
\]
Passing to the limit with $\ve$ we arrive at
\[
\min_{Q_{s_{min},T}} v \geq -M_{0}b^{\al-1}\max_{\tau\in [0,t]}\abs{s_{1}(\tau)-s_{2}(\tau)}.
\]
Combining the estimates for minimal and maximal value of $v$ we obtain
\[
\max_{Q_{s_{min},T}}\abs{v}\leq M_{0}b^{\al-1}\max_{\tau\in [0,t]}\abs{s_{1}(\tau)-s_{2}(\tau)}.
\]
Finally, we may estimate
\[
\abs{(Ps_{2})(t)-(Ps_{1})(t)} = \abs{\int_{0}^{s_{2}(t)}u_{2}(x,t)dx - \int_{0}^{s_{1}(t)}u_{1}(x,t)dx}
\]
\[
\leq \int_{0}^{s_{min}(t)}\abs{ v(x,t)} dx + \int_{s_{min}(t)}^{s_{max}(t)}u_{i}(x,t)dx
\]
\[
\leq s_{min}(t)\max_{Q_{s_{min},T}}\abs{v}+ (s_{max}(t)-s_{min}(t))^{2}M_{0}b^{\al-1}
\]
\[
\leq (b+MT)M_{0}b^{\al-1}\max_{\tau\in[0,t]}\abs{s_{1}(\tau) - s_{2}(\tau)} + M_{0}b^{\al-1}\max_{\tau\in[0,t]}\abs{s_{1}(\tau) - s_{2}(\tau)}^{2}.
\]
Thus $P$ is continuous and by the Schauder fixed point theorem there exist a fixed point of $P$. This way we proved the existence of the solution.
\end{proof}
In order to show that the obtained solution is unique we will prove the monotone dependence upon data.
\begin{theorem}\label{monot}
Let $(u^{i}, s_{i})$ be a solution to (\ref{Stefan}) given by Theorem \ref{final} corresponding to $b_{i}$ and $\uz^{i}$ for $i=1,2$. If $b_{1} \leq b_{2}$ and $\uz^{1} \leq \uz^{2}$, then for every $t\in[0,T]$ we have $s_{1}(t) \leq s_{2}(t)$.
\end{theorem}
\begin{proof}
The proof will be given in two steps. \\
1. Let us firstly discuss the case $b_{1} < b_{2}$, $\uzj \leq \uzd$ and $\uzj \not\equiv \uzd$ on $[0,b_{1}]$. We will proceed by contradiction. Let us assume that there exists $t\in[0,T]$ such that $\sj(t) > \sd(t)$. We denote $\tz = \inf\{t\in[0,T]: \sj(t) = \sd(t)\}$. Then by virtue of weak minimum principle (Lemma \ref{weakmax}) function $v = \ud-\uj$ is nonnegative in $Q_{\sj,\tz}$ and $v(\sj(\tz),\tz) = 0$. Thus, from Lemma \ref{maxda} and the proof of Lemma \ref{nonposda} we infer that either $v \equiv 0$ on $Q_{\sj,\tz}$ or $(\da v)(s(\tz),\tz) < 0$. The first possibility is a contradiction with $\uzj \not\equiv  \uzd$. Hence,
\[
0>(\da v)(s(\tz),\tz) = (\da \ud)(s(\tz),\tz) - (\da \uj)(s(\tz),\tz) = \dot{\sj}(\tz) - \dot{\sd}(\tz)
\]
and we obtain the contradiction with the definition of $\tz$. Thus, we obtain that if $b_{1} < b_{2}$, $\uzj \leq \uzd$ and $\uzj \not\equiv  \uzd$ on $[0,b_{1}]$, then $\sj(t) \leq \sd(t)$ for every $t \in [0,T]$.\\
2. In the general case, that is $b_{1} \leq b_{2} $ and $\uzj \leq \uzd$ we proceed as follows. We fix $\delta > 0$ and denote by $\uzde$ a smooth function defined on $[0,b_{2}+\delta]$ in such a way that $\uzde \equiv 0$ on $[b_{2}+\delta/2,b_{2}+\delta]$, $\uzde \geq \uzd$ on $[0,b_{2}]$ and $\max_{x\in[0,b_{2}]}(\uzde(x)-\uzd(x)) = \delta$, $\max_{x\in[b_{2},b_{2}+\delta/2]}\uzde(x) \leq \delta$. Then, we denote by $(\ude, \sde)$ the solution to (\ref{Stefan}) given by Theorem \ref{final} corresponding to $\uzde$. By the first step of the proof, we have $\sj \leq \sde$ and $\sd \leq \sde$. On the other hand performing calculations as in (\ref{pfo}) we have
\[
\sde(t) = b_{2}+\delta + \izt \dot{\sde}(\tau)d\tau = b_{2}+\delta - \izt (\da \ude)(\sde(\tau),\tau)d\tau
\]
\[
 = b_{2}+\delta +\int_{0}^{b_{2}+\delta}\uzde(x)dx - \int_{0}^{\sde(t)}\ude(x,t)dx
\]
and
\[
s_{2}(t) = b_{2} +\int_{0}^{b_{2}}\uzd(x)dx - \int_{0}^{\sd(t)}\ud(x,t)dx.
\]
Subtracting these identities we obtain
\[
\sde(t) - \sd(t) = \delta +\int_{0}^{b_{2}+\delta}\uzde(x)dx - \int_{0}^{b_{2}}\uzd(x)dx - \int_{0}^{\sde(t)}\ude(x,t)dx +  \int_{0}^{\sd(t)}\ud(x,t)dx
\]
\[
= \delta + \int_{0}^{b_{2}}\uzde(x) - \uzd(x)dx + \int_{b_{2}}^{b_{2}+\frac{\delta}{2}}\uzde(x)dx- \int_{\sd(t)}^{\sde(t)}\ude(x,t)dx - \int_{0}^{\sd(t)}[\ude(x,t) - \ud(x,t)]dx.
\]
The last two integrals are positive due to Lemma \ref{weakmax}. Making use of $\norm{\uzde-\uzd}_{L^{\infty}(0,b_{2})} = \delta$ we obtain
\[
\sj(t) \leq \sde(t) \leq \sd(t) + \delta + b_{2}\delta + \frac{\delta}{2}\delta \m{ for every }t \in [0,T].
\]
Passing to the limit with $\delta$ we obtain that $\sj(t) \leq \sd(t)$ for every $t \in [0,T]$.
\end{proof}
\begin{coro}
  From Theorem \ref{monot} it follows that the solution $(u,s)$ to problem (\ref{Stefan}) given by Theorem \ref{final} is unique.
\end{coro}
\section{Acknowledgments}
The author is grateful to Prof. Piotr Rybka and Dr. Adam Kubica for their valuable remarks.
The author was partly supported by National Sciences Center, Poland through 2017/26/M/ST1/00700 Grant.

\end{document}